\documentclass[a4paper,oneside,10pt]{article}%
\usepackage{amsmath}
\usepackage{amsfonts}
\usepackage{amssymb}
\usepackage{graphicx}
\usepackage{color}
\usepackage[square,numbers,sort&compress]{natbib}%
\providecommand{\U}[1]{\protect \rule{.1in}{.1in}}

\newtheorem{theorem}{Theorem}[section]

\newtheorem{corollary}[theorem]{Corollary}

\newtheorem{definition}[theorem]{Definition}

\newtheorem{example}[theorem]{Example}

\newtheorem{lemma}[theorem]{Lemma}

\newtheorem{proposition}[theorem]{Proposition}
\newtheorem{remark}[theorem]{Remark}

\newenvironment{proof}[1][Proof]{\noindent \textbf{#1.} }{\  $\Box$}
\numberwithin{equation}{section}

\begin{document}

\title{Limit theorems under nonlinear expectations dominated by sublinear expectations}
\author{Xiaojuan Li\thanks{Department of Mathematics, Qilu Normal University, Jinan
250200, China. lxj110055@126.com. }
\and Mingshang Hu \thanks{Zhongtai Securities Institute for Financial Studies,
Shandong University, Jinan, Shandong 250100, China. humingshang@sdu.edu.cn.
Research supported by NSF (No. 12326603, 11671231).} }
\maketitle

\textbf{Abstract}. In this paper, we obtain a new estimate for uniform integrability under sublinear expectations. Based on this, we establish the limit theorems under nonlinear expectations dominated by sublinear expectations through tightness, and the limit distributions can be completely nonlinear. Finally, we study the limit theorem in a special case, where the
limit distribution satisfies positive homogeneity.



\addcontentsline{toc}{section}{\hspace*{1.8em}Abstract}

\section{Introduction}

Let $(\Omega,\mathcal{H},\mathbb{\hat{E})}$ be a sublinear expectation space,
and let $\{(X_{n},Y_{n})\}_{n=1}^{\infty}$ be an i.i.d. sequence of
$2d$-dimensional random vectors, i.e., $(X_{n+1},Y_{n+1})\overset{d}{=}%
(X_{1},Y_{1})$ and $(X_{n+1},Y_{n+1})$ is independent of $(X_{1}$,$\ldots
$,$X_{n}$,$Y_{1}$,$\ldots$,$Y_{n})$ for $n\geq1$. Suppose that%
\begin{equation}
\mathbb{\hat{E}}[X_{1}]=\mathbb{\hat{E}}[-X_{1}]=0\text{ and }\lim
_{N\rightarrow \infty}\mathbb{\hat{E}}\left[  (|X_{1}|^{2}-N)^{+}%
+(|Y_{1}|-N)^{+}\right]  =0.\label{e0-0}%
\end{equation}
Peng \cite{P2008, P2010, P2019-1, P2019} showed that%
\begin{equation}
\lim_{n\rightarrow \infty}\mathbb{\hat{E}}\left[  \phi \left(  \frac{1}{\sqrt
{n}}\sum_{i=1}^{n}X_{i},\frac{1}{n}\sum_{i=1}^{n}Y_{i}\right)  \right]
=u^{\phi}(1,0,0)\label{e0-1}%
\end{equation}
for each $\phi \in C_{b.Lip}(\mathbb{R}^{2d})$, where $u^{\phi}$ is the unique
viscosity solution of the following PDE:%
\[
\partial_{t}u-G(D_{x}^{2}u,D_{y}u)=0\text{, }u(0,x,y)=\phi(x,y),
\]
and%
\[
G(A,p):=\mathbb{\hat{E}}\left[  \frac{1}{2}\langle AX_{1},X_{1}\rangle+\langle
p,Y_{1}\rangle \right]  \text{ for }(A,p)\in \mathbb{S}_{d}\times \mathbb{R}^{d}.
\]
Moreover, Chen \cite{Chen} obtained the strong laws of large numbers and Zhang
\cite{Zhang} studied the laws of the iterated logarithm under sublinear
expectations. The $\alpha$-stable limit theorem under sublinear expectations
was first obtained by Bayraktar and Munk in \cite{BM}, and the latest
developments can be found in \cite{HJLP, JL}. The central limit theorem with
mean uncertainty under sublinear expectations was studied in \cite{CE, GL}.
For the convergence rate of (\ref{e0-1}), the reader may refer to \cite{FP,
HL, Kr, Song, ZSG} and the references therein.

In some economic and financial problems, we sometimes need to consider convex
expectations or nonlinear expectations (see \cite{FS}). So it is natural to
consider the limit theorem under a nonlinear expectation, that is, replace the
sublinear expectation $\mathbb{\hat{E}}[\cdot]$ with the nonlinear expectation
$\mathbb{\tilde{E}}[\cdot]$ in (\ref{e0-1}). In \cite{Hu}, we found that the
limit distribution is characterized by the following PDE:%
\[
\partial_{t}u-\tilde{G}(D_{x}^{2}u,D_{y}u)=0\text{, }u(0,x,y)=\phi(x,y),
\]
where
\begin{equation}
\tilde{G}(A,p):=\lim_{n\rightarrow \infty}n\mathbb{\tilde{E}}\left[  \frac
{1}{n}\left(  \frac{1}{2}\langle AX_{1},X_{1}\rangle+\langle p,Y_{1}%
\rangle \right)  \right]  \text{ for }(A,p)\in \mathbb{S}_{d}\times
\mathbb{R}^{d}.\label{e0-2}%
\end{equation}
Unfortunately, $\tilde{G}(\cdot)$ satisfies positive homogeneity (see
Proposition \ref{pro-14}). In particular, $\tilde{G}(\cdot)$ is sublinear if
$\mathbb{\tilde{E}}[\cdot]$ is a convex expectation. Therefore, an interesting
problem is how to obtain the limit distribution that does not satisfy positive
homogeneity under nonlinear expectations.

Inspired by the relation (3.11) in the representation theorem for a kind of
consistent convex expectation in \cite{LH}, in order to obtain a more general
$\tilde{G}(\cdot)$, we must change the distribution of $(X_{n,1}%
,Y_{n,1}):=(\frac{1}{\sqrt{n}}X_{1},\frac{1}{n}Y_{1})$ under $\mathbb{\tilde
{E}}[\cdot]$ (see Examples \ref{ex-1} and \ref{ex-2}). So, we consider a
sequence of $2d$-dimensional random vectors $\{(X_{n,i},Y_{n,i}):n\in
\mathbb{N}$, $i=1$,$\ldots$,$n\}$. In order to apply the relationship between
tightness and weak compactness under nonlinear expectations introduced in
\cite{P2010}, we need a sublinear expectation $\mathbb{\hat{E}}[\cdot]$ that
dominates the nonlinear expectation $\mathbb{\tilde{E}}[\cdot]$, and we can
assume $(X_{n,1},Y_{n,1})\overset{d}{=}(\frac{1}{\sqrt{n}}X_{1},\frac{1}%
{n}Y_{1})$ under $\mathbb{\hat{E}}[\cdot]$. Under the assumption
\begin{equation}
\mathbb{\hat{E}}[|X_{1}|^{2+\gamma}+|Y_{1}|^{1+\gamma}]<\infty \text{ for some
positive constant }\gamma,\label{e0-4}%
\end{equation}
Peng used tightness to prove (\ref{e0-1}) in \cite{P2010}. In this paper, we
study the limit of $\mathbb{\tilde{E}}[\phi(\sum_{i=1}^{n}X_{n,i},\sum
_{i=1}^{n}Y_{n,i})]$ for each $\phi \in C_{b.Lip}(\mathbb{R}^{2d})$ under the
uniformly integrable condition (\ref{e0-0}), which is weaker than condition
(\ref{e0-4}).

The difficulty lies in establishing the uniformly integrable condition
(\ref{e2-10}) for the limit distribution of a subsequence of $(\sum_{i=1}%
^{n}X_{n,i},\sum_{i=1}^{n}Y_{n,i})$. By discovering a good property of
$\sum_{i<j}\langle X_{n,i},X_{n,j}\rangle$, we obtain estimate (\ref{e2-4}).
As far as we know, (\ref{e2-4}) is completely new in the literature. Based on
(\ref{e2-4}), we can obtain (\ref{e2-10}). Under the additional assumption
(\ref{e2-01}), we establish the limit theorem under the nonlinear expectation
$\mathbb{\tilde{E}}[\cdot]$.

Recently, Blessing and Kupper \cite{BK} first obtained
\begin{equation}
\lim_{n\rightarrow \infty}\frac{1}{n}\mathbb{\tilde{E}}\left[  n\phi \left(
\frac{1}{\sqrt{n}}\sum_{i=1}^{n}X_{i},\frac{1}{n}\sum_{i=1}^{n}Y_{i}\right)
\right]  =u^{\phi}(1,0,0)\label{e0-5}%
\end{equation}
for each $\phi \in C_{b.Lip}(\mathbb{R}^{2d})$, where $\mathbb{\tilde{E}}%
[\cdot]$ is a convex expectation and $u^{\phi}$ is the unique viscosity
solution of PDE (\ref{e2-25}). The convergence rate of (\ref{e0-5}) was
established in \cite{BJ}. By constructing a new nonlinear expectation
$\mathbb{\tilde{E}}_{1}[\cdot]$ satisfying%
\[
\mathbb{\tilde{E}}_{1}[\Phi(X_{n,1},\ldots,X_{n,n},Y_{n,1},\ldots
,Y_{n,n})]=\frac{1}{n}\mathbb{\tilde{E}}\left[  n\Phi \left(  \frac{1}{\sqrt
{n}}X_{1},\ldots,\frac{1}{\sqrt{n}}X_{n},\frac{1}{n}Y_{1},\ldots,\frac{1}%
{n}Y_{n}\right)  \right]
\]
for each $\Phi \in C_{b.Lip}(\mathbb{R}^{2nd})$ and applying Theorem
\ref{th-8}, we can obtain that (\ref{e0-5}) still holds true for any nonlinear
expectation $\mathbb{\tilde{E}}[\cdot]$ dominated by the sublinear expectation
$\mathbb{\hat{E}}[\cdot]$.

This paper is organized as follows. We recall some basic definitions and
results of nonlinear expectations in Section 2. In Section 3, we present and
prove the main limit theorems under nonlinear expectations dominated by the
sublinear expectation. In Section 4, we study the limit theorem in a special case, where the
limit distribution satisfies positive homogeneity.

\section{Preliminaries}

Let $\Omega$ be a given nonempty set. The set $\mathcal{H}$ is a linear space
of real valued functions defined on $\Omega$ satisfying $|\xi|\in \mathcal{H}$
if $\xi \in \mathcal{H}$ and $\phi(\xi_{1},\ldots,\xi_{n})\in \mathcal{H}$ for
each $n\in \mathbb{N}$, $\phi \in C_{b.Lip}(\mathbb{R}^{n})$, $\xi_{i}%
\in \mathcal{H}$ with $i\leq n$, where $C_{b.Lip}(\mathbb{R}^{n})$ denotes the
space of bounded Lipschitz functions on $\mathbb{R}^{n}$. $\mathbb{\tilde{E}%
}:\mathcal{H}\rightarrow \mathbb{R}$ is called a nonlinear expectation if it
satisfies $\mathbb{\tilde{E}}[\xi_{1}]\geq \mathbb{\tilde{E}}[\xi_{2}]$ if
$\xi_{1}\geq \xi_{2}$ and $\mathbb{\tilde{E}}[c]=c$ for each $c\in \mathbb{R}$.
A nonlinear expectation $\mathbb{\tilde{E}}[\cdot]$ is called a convex
expectation if it satisfies $\mathbb{\tilde{E}}[\rho \xi_{1}+(1-\rho)\xi
_{2}]\leq \rho \mathbb{\tilde{E}}[\xi_{1}]+(1-\rho)\mathbb{\tilde{E}}[\xi_{2}]$
for each $\rho \in \lbrack0,1]$. A nonlinear expectation $\mathbb{\hat{E}}%
[\cdot]$ is called a sublinear expectation if it satisfies $\mathbb{\hat{E}%
}[\xi_{1}+\xi_{2}]$ $\leq \mathbb{\hat{E}}[\xi_{1}]+\mathbb{\hat{E}}[\xi_{2}]$
and $\mathbb{\hat{E}}[\lambda \xi]=\lambda \mathbb{\hat{E}}[\xi]$ for each
$\lambda \geq0$. The triple $(\Omega,\mathcal{H},\mathbb{\tilde{E}})$ (resp.
$(\Omega,\mathcal{H},\mathbb{\hat{E}})$) is called a nonlinear (resp.
sublinear) expectation space. The statement that a nonlinear expectation
$\mathbb{\tilde{E}}[\cdot]$ is dominated by a sublinear expectation
$\mathbb{\hat{E}}[\cdot]$ means that%
\begin{equation}
\mathbb{\tilde{E}}[\xi_{1}]-\mathbb{\tilde{E}}[\xi_{2}]\leq \mathbb{\hat{E}%
}[\xi_{1}-\xi_{2}]\text{ for each }\xi_{1}\text{, }\xi_{2}\in \mathcal{H}.
\label{e1-1}%
\end{equation}

Let $\mathbb{\hat{E}}[\cdot]$ be a sublinear expectation on $(\Omega
,\mathcal{H})$. The space $L^{1}(\Omega)$ is the completion of $\mathcal{H}$
under the norm $||\xi||:=\mathbb{\hat{E}}[|\xi|]$, and we set $L^{p}%
(\Omega)=\{ \xi \in L^{1}(\Omega):|\xi|^{p}\in L^{1}(\Omega)\}$ for $p\geq1$.
If a nonlinear expectation $\mathbb{\tilde{E}}[\cdot]$ is dominated by
$\mathbb{\hat{E}}[\cdot]$, then $\mathbb{\hat{E}}[\cdot]$ and $\mathbb{\tilde
{E}}[\cdot]$ can be continuously extended to $L^{1}(\Omega)$, and (\ref{e1-1})
still holds. The following important property is Proposition 1.3.8 in
\cite{P2019}.

\begin{proposition}
\label{pro-1}Suppose that a nonlinear expectation $\mathbb{\tilde{E}}[\cdot]$
is dominated by a sublinear expectation $\mathbb{\hat{E}}[\cdot]$ on
$(\Omega,\mathcal{H})$. Let $X\in L^{1}(\Omega)$ satisfy $\mathbb{\hat{E}%
}[X]=-\mathbb{\hat{E}}[-X]$. Then%
\begin{equation}
\mathbb{\tilde{E}}[\alpha X+Y]=\alpha \mathbb{\hat{E}}[X]+\mathbb{\tilde{E}%
}[Y]\text{ for each }Y\in L^{1}(\Omega)\text{, }\alpha \in \mathbb{R}.
\label{e1-2}%
\end{equation}

\end{proposition}

For each given $d$-dimensional random vector $X=(X^{1}$,$\ldots$,$X^{d})^{T}$
with $X^{i}\in L^{1}(\Omega)$ for $i\leq d$, the distribution of $X$ under a
nonlinear expectation $\mathbb{\tilde{E}}[\cdot]$ (resp. sublinear expectation
$\mathbb{\hat{E}}[\cdot]$) is defined by%
\[
\mathbb{\tilde{F}}_{X}[\phi]:=\mathbb{\tilde{E}}[\phi(X)]\text{ (resp.
}\mathbb{\hat{F}}_{X}[\phi]:=\mathbb{\hat{E}}[\phi(X)]\text{) for each }%
\phi \in C_{b.Lip}(\mathbb{R}^{d}).
\]
It is easy to check that $\mathbb{\tilde{F}}_{X}[\cdot]$ is a nonlinear
expectation on $(\mathbb{R}^{d},C_{b.Lip}(\mathbb{R}^{d}))$, and
$\mathbb{\hat{F}}_{X}[\cdot]$ is a sublinear expectation on $(\mathbb{R}%
^{d},C_{b.Lip}(\mathbb{R}^{d}))$. Two $d$-dimensional random vectors $X$ and
$Y$ are called identically distributed, denoted by $X\overset{d}{=}Y$, under
$\mathbb{\tilde{E}}[\cdot]$ (resp. $\mathbb{\hat{E}}[\cdot]$) if
$\mathbb{\tilde{F}}_{X}[\cdot]=\mathbb{\tilde{F}}_{Y}[\cdot]$ (resp.
$\mathbb{\hat{F}}_{X}[\cdot]=\mathbb{\hat{F}}_{Y}[\cdot]$). A $d$-dimensional
random vector $Y$ is said to be independent of a $d^{\prime}$-dimensional
random vector random vector $X$ under $\mathbb{\tilde{E}}[\cdot]$ (resp.
$\mathbb{\hat{E}}[\cdot]$) if%
\[
\mathbb{\tilde{E}}[\phi(X,Y)]=\mathbb{\tilde{E}}[\mathbb{\tilde{E}}%
[\phi(x,Y)]_{x=X}]\text{ (resp. }\mathbb{\hat{E}}[\phi(X,Y)]=\mathbb{\hat{E}%
}\left[  \mathbb{\hat{E}}[\phi(x,Y)]_{x=X}\right]  \text{) for each }\phi \in
C_{b.Lip}(\mathbb{R}^{d^{\prime}+d}).
\]

\begin{remark}
In the above definition of distribution, we only need $\phi(X)\in L^{1}%
(\Omega)$ for each $\phi \in C_{b.Lip}(\mathbb{R}^{d})$, while $X^{i}$ may not
in $L^{1}(\Omega)$ for $i\leq d$. Similar requirements for the definition of
independence. If $X^{i}\in L^{1}(\Omega)$ for $i\leq d$, we define
$\mathbb{\tilde{E}}[X]=(\mathbb{\tilde{E}}[X^{1}]$,$\ldots$,$\mathbb{\tilde
{E}}[X^{d}])^{T}$.
\end{remark}

\begin{definition}
\label{de-2}A sequence of $d$-dimensional random vectors $\{X_{i}%
\}_{i=1}^{\infty}$ defined respectively on nonlinear expectation spaces
$(\Omega_{i},\mathcal{H}_{i},\mathbb{\tilde{E}}_{i})$ (resp. sublinear
expectation spaces $(\Omega_{i},\mathcal{H}_{i},\mathbb{\hat{E}}_{i})$).
$\{X_{i}\}_{i=1}^{\infty}$ is said to converge in distribution if
$\mathbb{\tilde{F}}_{X_{i}}[\phi]:=\mathbb{\tilde{E}}_{i}[\phi(X_{i})]$ (resp.
$\mathbb{\hat{F}}_{X_{i}}[\phi]:=\mathbb{\hat{E}}_{i}[\phi(X_{i})]$) converges
as $i\rightarrow \infty$ for each $\phi \in C_{b.Lip}(\mathbb{R}^{d})$.
\end{definition}

\begin{remark}
Define
\[
\mathbb{\tilde{F}}[\phi]=\lim_{i\rightarrow \infty}\mathbb{\tilde{F}}_{X_{i}%
}[\phi]\text{ (resp. }\mathbb{\hat{F}}[\phi]=\lim_{i\rightarrow \infty
}\mathbb{\hat{F}}_{X_{i}}[\phi]\text{) for each }\phi \in C_{b.Lip}%
(\mathbb{R}^{d}).
\]
It is easy to verify that $\mathbb{\tilde{F}}[\cdot]$ (resp. $\mathbb{\hat{F}%
}[\cdot]$) is a nonlinear (resp. sublinear) expectation on $(\mathbb{R}%
^{d},C_{b.Lip}(\mathbb{R}^{d}))$. $\{X_{i}\}_{i=1}^{\infty}$ is said to
converge in distribution to $X$ defined on a nonlinear expectation space
$(\Omega,\mathcal{H},\mathbb{\tilde{E}})$ if $\mathbb{\tilde{F}}_{X}%
[\cdot]=\mathbb{\tilde{F}}[\cdot]$.
\end{remark}

\begin{definition}
\label{de-3}Let $\mathbb{\hat{F}}[\cdot]$ be a sublinear expectation on
$(\mathbb{R}^{d},C_{b.Lip}(\mathbb{R}^{d}))$. $\mathbb{\hat{F}}[\cdot]$ is
called tight if, for any $\varepsilon>0$, there exist $\phi \in C_{b.Lip}%
(\mathbb{R}^{d})$ and $N>0$ satisfying $I_{\{|x|\geq N\}}\leq \phi$ and
$\mathbb{\hat{F}}[\phi]<\varepsilon$.
\end{definition}

\begin{definition}
\label{de-4}Let $\{ \mathbb{\tilde{F}}_{i}[\cdot]\}_{i=1}^{\infty}$ be a
sequence of nonlinear expectations on $(\mathbb{R}^{d},C_{b.Lip}%
(\mathbb{R}^{d}))$. $\{ \mathbb{\tilde{F}}_{i}[\cdot]\}_{i=1}^{\infty}$ is
called tight if there exists a tight sublinear expectation $\mathbb{\hat{F}%
}[\cdot]$ on $(\mathbb{R}^{d},C_{b.Lip}(\mathbb{R}^{d}))$ such that
$\mathbb{\tilde{F}}_{i}[\cdot]$ is dominated by $\mathbb{\hat{F}}[\cdot]$ for
$i\geq1$.
\end{definition}

\begin{lemma}
\label{le-5}Let $\mathbb{\hat{F}}[\cdot]$ be a sublinear expectation on
$(\mathbb{R}^{d},C_{b.Lip}(\mathbb{R}^{d}))$. If $\beta:=\sup \{ \mathbb{\hat
{F}}[|x|\wedge i]:i\geq1\}<\infty$, then $\mathbb{\hat{F}}[\cdot]$ is tight.
\end{lemma}

\begin{proof}
For any $\varepsilon>0$, we can choose $N>1$ such that $(N-1)^{-1}%
\beta<\varepsilon$. It is clear that there exists a $\phi \in C_{b.Lip}%
(\mathbb{R}^{d})$ satisfying $I_{\{|x|\geq N\}}\leq \phi \leq I_{\{|x|\geq
N-1\}}$. Then%
\[
\mathbb{\hat{F}}[\phi]\leq \mathbb{\hat{F}}[(N-1)^{-1}(|x|\wedge N)]\leq
(N-1)^{-1}\beta<\varepsilon,
\]
which implies that $\mathbb{\hat{F}}[\cdot]$ is tight.
\end{proof}

The following fundamental property is Theorem 9 in \cite{P2010}.

\begin{theorem}
\label{th-6}Let $\{X_{i}\}_{i=1}^{\infty}$ be a sequence of $d$-dimensional
random vectors defined respectively on nonlinear expectation spaces
$(\Omega_{i},\mathcal{H}_{i},\mathbb{\tilde{E}}_{i})$. Suppose that $\{
\mathbb{\tilde{F}}_{X_{i}}[\cdot]\}_{i=1}^{\infty}$ is tight. Then there
exists a subsequence $\{X_{i_{j}}\}_{j=1}^{\infty}$ which converges in distribution.
\end{theorem}

The following property is useful in proving the limit theorems under nonlinear expectations.

\begin{proposition}
\label{pro-7}Let $\{X_{i}=(X_{i,1},\ldots,X_{i,m})\}_{i=1}^{\infty}$ be a
sequence of $d\times m$-dimensional random vectors defined respectively on
nonlinear expectation spaces $(\Omega_{i},\mathcal{H}_{i},\mathbb{\tilde{E}%
}_{i})$ such that $X_{i,k}$ is independent of $(X_{i,1},\ldots,X_{i,k-1})$ for
$k=2$,$\ldots$,$m$. Let $Y=(Y_{1},\ldots,Y_{m})$ be a $d\times m$-dimensional
random vector defined on a nonlinear expectation spaces $(\Omega
,\mathcal{H},\mathbb{\tilde{E}})$ such that $Y_{k}$ is independent of
$(Y_{1},\ldots,Y_{k-1})$ for $k=2$,$\ldots$,$m$. Suppose that $\{X_{i,k}%
\}_{i=1}^{\infty}$ converges in distribution to $Y_{k}$ and $\{ \mathbb{\tilde
{F}}_{X_{i,k}}[\cdot]\}_{i=1}^{\infty}$ is tight for $k=1$,$\ldots$,$m$. Then
$\{X_{i}\}_{i=1}^{\infty}$ converges in distribution to $Y$.
\end{proposition}

\begin{proof}
We only prove the case $m=2$. The proof for $m>2$ is similar. For each given
$\phi \in C_{b.Lip}(\mathbb{R}^{2d})$, define%
\[
\psi_{i}(x)=\mathbb{\tilde{E}}_{i}[\phi(x,X_{i,2})]\text{ and }\psi
(x)=\mathbb{\tilde{E}}[\phi(x,Y_{2})]\text{ for }x\in \mathbb{R}^{d}.
\]
Since $\{X_{i,2}\}_{i=1}^{\infty}$ converges in distribution to $Y_{2}$, we
have
\begin{equation}
\psi_{i}(x)\rightarrow \psi(x)\text{ as }i\rightarrow \infty \text{ for each
}x\in \mathbb{R}^{d}. \label{e1-4}%
\end{equation}
Note that $\{ \mathbb{\tilde{F}}_{X_{i,2}}[\cdot]\}_{i=1}^{\infty}$ and $\{
\mathbb{\tilde{F}}_{X_{i,1}}[\cdot]\}_{i=1}^{\infty}$ are tight. Then there
exists a tight sublinear expectation $\mathbb{\hat{F}}[\cdot]$ on
$(\mathbb{R}^{d},C_{b.Lip}(\mathbb{R}^{d}))$ such that $\mathbb{\tilde{F}%
}_{X_{i,2}}[\cdot]$ and $\mathbb{\tilde{F}}_{X_{i,1}}[\cdot]$ are dominated by
$\mathbb{\hat{F}}[\cdot]$ for $i\geq1$. Thus we get%
\begin{equation}
|\psi_{i}(x)|=|\mathbb{\tilde{F}}_{X_{i,2}}[\phi(x,\cdot)]|\leq \mathbb{\hat
{F}}[|\phi(x,\cdot)|]\leq||\phi||_{\infty}, \label{e1-5}%
\end{equation}%
\begin{equation}
|\psi_{i}(x)-\psi_{i}(x^{\prime})|=|\mathbb{\tilde{F}}_{X_{i,2}}[\phi
(x,\cdot)]-\mathbb{\tilde{F}}_{X_{i,2}}[\phi(x^{\prime},\cdot)]|\leq
\mathbb{\hat{F}}[|\phi(x,\cdot)-\phi(x^{\prime},\cdot)|]\leq L_{\phi
}|x-x^{\prime}| \label{e1-3}%
\end{equation}
for $x$, $x^{\prime}\in \mathbb{R}^{d}$ and $i\geq1$, where $||\phi||_{\infty
}:=\sup \{|\phi(x_{1},x_{2})|:x_{1}$, $x_{2}\in \mathbb{R}^{d}\}$ and $L_{\phi}$
is the Lipschitz constant of $\phi$. For any $\varepsilon>0$, there exist
$\varphi \in C_{b.Lip}(\mathbb{R}^{d})$ and $N>0$ satisfying $I_{\{|x|\geq
N\}}\leq \varphi$ and $\mathbb{\hat{F}}[\varphi]<\varepsilon$. It follows from
(\ref{e1-4}) and (\ref{e1-3}) that%
\begin{equation}
\gamma_{i}^{N}:=\sup_{|x|\leq N}|\psi_{i}(x)-\psi(x)|\rightarrow0\text{ as
}i\rightarrow \infty. \label{e1-6}%
\end{equation}
By (\ref{e1-5}) and (\ref{e1-6}), we obtain%
\[
|\mathbb{\tilde{E}}_{i}[\psi_{i}(X_{i,1})]-\mathbb{\tilde{E}}_{i}[\psi
(X_{i,1})]|\leq \mathbb{\hat{F}}[|\psi_{i}(\cdot)-\psi(\cdot)|]\leq
\mathbb{\hat{F}}[\gamma_{i}^{N}+2||\phi||_{\infty}\varphi]\leq \gamma_{i}%
^{N}+2||\phi||_{\infty}\varepsilon,
\]
which implies $|\mathbb{\tilde{E}}_{i}[\psi_{i}(X_{i,1})]-\mathbb{\tilde{E}%
}_{i}[\psi(X_{i,1})]|\rightarrow0$ as $i\rightarrow \infty$ by taking
$i\rightarrow \infty$ and then $\varepsilon \downarrow0$. Thus%
\[
\lim_{i\rightarrow \infty}\mathbb{\tilde{E}}_{i}[\phi(X_{i,1},X_{i,2}%
)]=\lim_{i\rightarrow \infty}\mathbb{\tilde{E}}_{i}[\psi_{i}(X_{i,1}%
)]=\lim_{i\rightarrow \infty}\mathbb{\tilde{E}}_{i}[\psi(X_{i,1}%
)]=\mathbb{\tilde{E}}[\psi(Y_{1})]=\mathbb{\tilde{E}}[\phi(Y_{1},Y_{2})],
\]
which implies the desired result.
\end{proof}

\section{Limit theorems under nonlinear expectations}

\begin{theorem}
\label{th-8}Let a nonlinear expectation $\mathbb{\tilde{E}}[\cdot]$ be
dominated by a sublinear expectation $\mathbb{\hat{E}}[\cdot]$ on
$(\Omega,\mathcal{H})$, and let $\{(X_{n,i},Y_{n,i}):n\in \mathbb{N}$,
$i=1$,$\ldots$,$n\}$ be a sequence of $2d$-dimensional random vectors such
that $(X_{n,i},Y_{n,i})\overset{d}{=}(X_{n,1},Y_{n,1})$ and $(X_{n,i}%
,Y_{n,i})$ is independent of $(X_{n,1}$,$\ldots$,$X_{n,i-1}$,$Y_{n,1}$%
,$\ldots$,$Y_{n,i-1})$ for $n\in \mathbb{N}$, $i=2$,$\ldots$,$n$ under
$\mathbb{\tilde{E}}[\cdot]$ and $\mathbb{\hat{E}}[\cdot]$ respectively.
Suppose that%
\begin{equation}
\mathbb{\hat{E}}[X_{n,1}]=\mathbb{\hat{E}}[-X_{n,1}]=0\text{ for }%
n\in \mathbb{N},\text{ }\lim_{N\rightarrow \infty}\sup_{n\geq1}\mathbb{\hat{E}%
}\left[  (n|X_{n,1}|^{2}-N)^{+}+(n|Y_{n,1}|-N)^{+}\right]  =0, \label{e2-00}%
\end{equation}
and for each given $(A,p)\in \mathbb{S}_{d}\times \mathbb{R}^{d}$,
\begin{equation}
\tilde{G}(A,p):=\lim_{n\rightarrow \infty}n\mathbb{\tilde{E}}\left[  \frac
{1}{2}\langle AX_{n,1},X_{n,1}\rangle+\langle p,Y_{n,1}\rangle \right]  \text{
exists,} \label{e2-01}%
\end{equation}
where $\mathbb{S}_{d}$ denotes the space of $d\times d$ symmetric matrices.
Then for each $\phi \in C_{b.Lip}(\mathbb{R}^{2d})$, we have%
\begin{equation}
\lim_{n\rightarrow \infty}\mathbb{\tilde{E}}\left[  \phi \left(  \sum_{i=1}%
^{n}X_{n,i},\sum_{i=1}^{n}Y_{n,i}\right)  \right]  =u^{\phi}(1,0,0),
\label{e2-1}%
\end{equation}
where $u^{\phi}$ is the unique viscosity solution of the following PDE:%
\begin{equation}
\partial_{t}u-\tilde{G}(D_{x}^{2}u,D_{y}u)=0\text{, }u(0,x,y)=\phi(x,y).
\label{e2-0}%
\end{equation}

\end{theorem}

In order to prove the above theorem, we need the following lemmas.

\begin{lemma}
\label{le-9}Let $(\Omega,\mathcal{H},\mathbb{\hat{E}})$ be a sublinear
expectation space. Suppose that $\eta \in L^{1}(\Omega)$ is independent of
$\xi \in L^{1}(\Omega)$ and $\mathbb{\hat{E}}\left[  (|\xi|-N)^{+}%
+(|\eta|-N)^{+}\right]  \rightarrow0$ as $N\rightarrow \infty$. Then $\xi
\eta \in L^{1}(\Omega)$ and $\mathbb{\hat{E}}\left[  \xi \eta \right]
=\mathbb{\hat{E}}\left[  \mathbb{\hat{E}}\left[  x\eta \right]  _{x=\xi
}\right]  $.
\end{lemma}

\begin{proof}
Set $\xi^{N}=(\xi \wedge N)\vee(-N)$ and $\eta^{N}=(\eta \wedge N)\vee(-N)$. It
is easy to check that $(|\xi|-N)^{+}=|\xi-\xi^{N}|$ and $(|\eta|-N)^{+}%
=|\eta-\eta^{N}|$. Since $\eta$ is independent of $\xi$, we have%
\[
\mathbb{\hat{E}}\left[  |\xi^{N_{1}}\eta^{N_{2}^{\prime}}-\xi^{N_{1}}%
\eta^{N_{2}}|\right]  =\mathbb{\hat{E}}\left[  |\xi^{N_{1}}|\right]
\mathbb{\hat{E}}\left[  |\eta^{N_{2}^{\prime}}-\eta^{N_{2}}|\right]
\leq \mathbb{\hat{E}}\left[  |\xi|\right]  \mathbb{\hat{E}}\left[
(|\eta|-N_{2})^{+}\right]
\]
for each $N_{1}>0$ and $N_{2}^{\prime}>N_{2}>0$. Thus $\xi^{N_{1}}\eta \in
L^{1}(\Omega)$ for each $N_{1}>0$. Similarly, we can get%
\[
\mathbb{\hat{E}}\left[  |\xi^{N_{1}^{\prime}}\eta-\xi^{N_{1}}\eta|\right]
\leq \mathbb{\hat{E}}\left[  (|\xi|-N_{1})^{+}\right]  \mathbb{\hat{E}}\left[
|\eta|\right]
\]
for $N_{1}^{\prime}>N_{1}>0$. Thus $\xi \eta \in L^{1}(\Omega)$. Define
$\psi_{N}(x)=\mathbb{\hat{E}}\left[  x\eta^{N}\right]  $ and $\psi
(x)=\mathbb{\hat{E}}\left[  x\eta \right]  $ for $x\in \mathbb{R}$ and $N>0$, we
have%
\[
|\psi_{N}(x)-\psi(x)|\leq \mathbb{\hat{E}}\left[  |x\eta^{N}-x\eta|\right]
=|x|\mathbb{\hat{E}}\left[  (|\eta|-N)^{+}\right]  \text{ and }|\psi
(x)-\psi(x^{\prime})|\leq|x-x^{\prime}|\mathbb{\hat{E}}\left[  |\eta|\right]
.
\]
Then we obtain%
\begin{align*}
|\mathbb{\hat{E}}\left[  \psi_{N}(\xi^{N})\right]  -\mathbb{\hat{E}}\left[
\psi(\xi)\right]  |  &  \leq \mathbb{\hat{E}}\left[  |\psi_{N}(\xi^{N}%
)-\psi(\xi^{N})|+|\psi(\xi^{N})-\psi(\xi)|\right] \\
&  \leq \mathbb{\hat{E}}\left[  |\xi|\right]  \mathbb{\hat{E}}\left[
(|\eta|-N)^{+}\right]  +\mathbb{\hat{E}}\left[  (|\xi|-N)^{+}\right]
\mathbb{\hat{E}}\left[  |\eta|\right]  .
\end{align*}
Since $\mathbb{\hat{E}}\left[  \xi^{N}\eta^{N}\right]  =\mathbb{\hat{E}%
}\left[  \psi_{N}(\xi^{N})\right]  $ and $\mathbb{\hat{E}}\left[  \xi^{N}%
\eta^{N}\right]  \rightarrow \mathbb{\hat{E}}\left[  \xi \eta \right]  $ as
$N\rightarrow \infty$, we deduce $\mathbb{\hat{E}}\left[  \xi \eta \right]
=\mathbb{\hat{E}}\left[  \psi(\xi)\right]  $.
\end{proof}

\begin{lemma}
\label{le-10}Let the assumptions in Theorem \ref{th-8} hold. Then, for each
$(A,p)\in \mathbb{S}_{d}\times \mathbb{R}^{d}$ and $N>0$, we have%
\begin{equation}
\mathbb{\hat{E}}\left[  \frac{1}{2}\langle AS_{n,k}^{X},S_{n,k}^{X}%
\rangle+\langle p,S_{n,k}^{Y}\rangle \right]  =k\mathbb{\hat{E}}\left[
\frac{1}{2}\langle AX_{n,1},X_{n,1}\rangle+\langle p,Y_{n,1}\rangle \right]  ,
\label{e2-2}%
\end{equation}%
\begin{equation}
\mathbb{\tilde{E}}\left[  \frac{1}{2}\langle AS_{n,k}^{X},S_{n,k}^{X}%
\rangle+\langle p,S_{n,k}^{Y}\rangle \right]  =k\mathbb{\tilde{E}}\left[
\frac{1}{2}\langle AX_{n,1},X_{n,1}\rangle+\langle p,Y_{n,1}\rangle \right]  ,
\label{e2-3}%
\end{equation}%
\begin{equation}
\mathbb{\hat{E}}\left[  (|S_{n,k}^{X}|^{2}-N)^{+}+(|S_{n,k}^{Y}|-N)^{+}%
\right]  \leq \mathbb{\hat{E}}\left[  \left(  k|X_{n,1}|^{2}-\frac{N}%
{2}\right)  ^{+}+(k|Y_{n,1}|-N)^{+}\right]  +\frac{4\left(  k\mathbb{\hat{E}%
}\left[  |X_{n,1}|^{2}\right]  \right)  ^{2}}{N}, \label{e2-4}%
\end{equation}
where $S_{n,k}^{X}=\sum_{i=1}^{k}X_{n,i}$ and $S_{n,k}^{Y}=\sum_{i=1}%
^{k}Y_{n,i}$ for $k\leq n$.
\end{lemma}

\begin{proof}
By Lemma \ref{le-9}, we have $\mathbb{\hat{E}}\left[  \langle X_{n,i}%
,X_{n,j}\rangle \right]  =\mathbb{\hat{E}}\left[  -\langle X_{n,i}%
,X_{n,j}\rangle \right]  =0$ for $i\not =j$. Then, by Proposition \ref{pro-1},
we get%
\begin{equation}
\mathbb{\hat{E}}\left[  \frac{1}{2}\langle AS_{n,k}^{X},S_{n,k}^{X}%
\rangle+\langle p,S_{n,k}^{Y}\rangle \right]  =\mathbb{\hat{E}}\left[
\sum_{i=1}^{k}\left(  \frac{1}{2}\langle AX_{n,i},X_{n,i}\rangle+\langle
p,Y_{n,i}\rangle \right)  \right]  \label{e2-5}%
\end{equation}
and%
\begin{equation}
\mathbb{\tilde{E}}\left[  \frac{1}{2}\langle AS_{n,k}^{X},S_{n,k}^{X}%
\rangle+\langle p,S_{n,k}^{Y}\rangle \right]  =\mathbb{\tilde{E}}\left[
\sum_{i=1}^{k}\left(  \frac{1}{2}\langle AX_{n,i},X_{n,i}\rangle+\langle
p,Y_{n,i}\rangle \right)  \right]  . \label{e2-6}%
\end{equation}
Noting that $(X_{n,i},Y_{n,i})\overset{d}{=}(X_{n,1},Y_{n,1})$ and
$(X_{n,i},Y_{n,i})$ is independent of $(X_{n,1}$,$\ldots$,$X_{n,i-1}$%
,$Y_{n,1}$,$\ldots$,$Y_{n,i-1})$ for $i=2$,$\ldots$,$n$ under $\mathbb{\tilde
{E}}[\cdot]$ and $\mathbb{\hat{E}}[\cdot]$ respectively, we obtain
(\ref{e2-2}) and (\ref{e2-3}) from (\ref{e2-5}) and (\ref{e2-6}). Since
$(a+b)^{+}\leq a^{+}+b^{+}$ and%
\[
|S_{n,k}^{X}|^{2}=\sum_{i=1}^{k}|X_{n,i}|^{2}+2\sum_{1\leq i<j\leq k}\langle
X_{n,i},X_{n,j}\rangle=\sum_{i=1}^{k}|X_{n,i}|^{2}+2\sum_{i=1}^{k-1}\langle
S_{n,i}^{X},X_{n,i+1}\rangle,
\]
we deduce%
\begin{align*}
(|S_{n,k}^{X}|^{2}-N)^{+}+(|S_{n,k}^{Y}|-N)^{+}  &  \leq \sum_{i=1}^{k}\left[
\left(  |X_{n,i}|^{2}-\frac{N}{2k}\right)  ^{+}+\left(  |Y_{n,i}|-\frac{N}%
{k}\right)  ^{+}\right] \\
&  \  \  \  \ +2\left(  \sum_{i=1}^{k-1}\langle S_{n,i}^{X},X_{n,i+1}%
\rangle-\frac{N}{4}\right)  ^{+}.
\end{align*}
Thus%
\begin{align*}
\mathbb{\hat{E}}\left[  (|S_{n,k}^{X}|^{2}-N)^{+}+(|S_{n,k}^{Y}|-N)^{+}%
\right]   &  \leq k\mathbb{\hat{E}}\left[  \left(  |X_{n,1}|^{2}-\frac{N}%
{2k}\right)  ^{+}+\left(  |Y_{n,1}|-\frac{N}{k}\right)  ^{+}\right] \\
&  \  \  \  \ +\frac{8}{N}\mathbb{\hat{E}}\left[  \left \vert \sum_{i=1}%
^{k-1}\langle S_{n,i}^{X},X_{n,i+1}\rangle \right \vert ^{2}\right]  .
\end{align*}
It is easy to verify that $k(|X_{n,1}|^{2}-\frac{N}{2k})^{+}=(k|X_{n,1}%
|^{2}-\frac{N}{2})^{+}$. By Lemma \ref{le-9}, we have $|X_{n,i}|^{2}%
|X_{n,j}|^{2}\in L^{1}(\Omega)$ for $i\not =j$. Note that
\[
\mathbb{\hat{E}}\left[  \langle S_{n,i}^{X},X_{n,i+1}\rangle \langle
S_{n,j}^{X},X_{n,j+1}\rangle \right]  =\mathbb{\hat{E}}\left[  -\langle
S_{n,i}^{X},X_{n,i+1}\rangle \langle S_{n,j}^{X},X_{n,j+1}\rangle \right]
=0\text{ for }i\not =j.
\]
Then%
\[
\mathbb{\hat{E}}\left[  \left \vert \sum_{i=1}^{k-1}\langle S_{n,i}%
^{X},X_{n,i+1}\rangle \right \vert ^{2}\right]  =\mathbb{\hat{E}}\left[
\sum_{i=1}^{k-1}|\langle S_{n,i}^{X},X_{n,i+1}\rangle|^{2}\right]  \leq
\sum_{i=1}^{k-1}\mathbb{\hat{E}}[|S_{n,i}^{X}|^{2}|X_{n,i+1}|^{2}]\leq
\frac{\left(  k\mathbb{\hat{E}}\left[  |X_{n,1}|^{2}\right]  \right)  ^{2}}%
{2}.
\]
Thus we obtain (\ref{e2-4}).
\end{proof}

\noindent \textbf{Proof of Theorem \ref{th-8}.} For each $n\in \mathbb{N}$,
define $\mathbb{\tilde{F}}_{n}[\cdot]$ and $\mathbb{\hat{F}}_{n}[\cdot]$ on
$(\mathbb{R}^{2d},C_{b.Lip}(\mathbb{R}^{2d}))$ as follows: for $\phi \in
C_{b.Lip}(\mathbb{R}^{2d})$,
\[
\mathbb{\tilde{F}}_{n}[\phi]:=\mathbb{\tilde{E}}\left[  \phi \left(  \sum
_{i=1}^{n}X_{n,i},\sum_{i=1}^{n}Y_{n,i}\right)  \right]  ,\text{ }%
\mathbb{\hat{F}}_{n}[\phi]:=\mathbb{\hat{E}}\left[  \phi \left(  \sum_{i=1}%
^{n}X_{n,i},\sum_{i=1}^{n}Y_{n,i}\right)  \right]  .
\]
Set%
\[
\mathbb{\hat{F}}[\phi]=\sup_{n\geq1}\mathbb{\hat{F}}_{n}[\phi]\text{ for }%
\phi \in C_{b.Lip}(\mathbb{R}^{2d}).
\]
It is easy to verify that $\mathbb{\hat{F}}[\cdot]$ is a sublinear expectation
on $(\mathbb{R}^{d},C_{b.Lip}(\mathbb{R}^{d}))$ and $\mathbb{\tilde{F}}%
_{n}[\cdot]$ is dominated by $\mathbb{\hat{F}}[\cdot]$ for $n\geq1$. Note that%
\[
\mathbb{\hat{E}}\left[  |S_{n,n}^{X}|^{2}+|S_{n,n}^{Y}|\right]  \leq
n\mathbb{\hat{E}}\left[  |X_{n,1}|^{2}+|Y_{n,1}|\right]  \leq2N+\mathbb{\hat
{E}}\left[  (n|X_{n,1}|^{2}-N)^{+}+(n|Y_{n,1}|-N)^{+}\right]  ,
\]
where $S_{n,n}^{X}$ and $S_{n,n}^{Y}$ are defined in Lemma \ref{le-10}. Then
we obtain%
\begin{equation}
\sup_{n\geq1}\mathbb{\hat{E}}\left[  |S_{n,n}^{X}|^{2}+|S_{n,n}^{Y}|\right]
\leq \sup_{n\geq1}n\mathbb{\hat{E}}\left[  |X_{n,1}|^{2}+|Y_{n,1}|\right]
<\infty \label{e2-7}%
\end{equation}
by (\ref{e2-00}). From this we can easily get $\mathbb{\hat{F}}[|x|]<\infty$,
which implies that $\mathbb{\hat{F}}[\cdot]$ is tight by Lemma \ref{le-5}.

Set $\bar{\Omega}=(\mathbb{R}^{2d})^{[0,1]}$, the canonical process is
$B_{t}(\omega)=(\tilde{B}_{t}(\omega),\bar{B}_{t}(\omega)):=\omega_{t}$ for
$t\in \lbrack0,1]$ and $\omega \in \bar{\Omega}$, where $\tilde{B}$ and $\bar{B}$
are $d$-dimensional processes. For each given nonnegative integer $m$, set
$\mathcal{D}_{m}=\{t_{j}^{m}=j2^{-m}:j=0$,$\ldots$,$2^{m}\}$ and define%
\[
\mathcal{\bar{H}}_{m}=\{ \phi(B_{t_{1}^{m}}-B_{t_{0}^{m}},B_{t_{2}^{m}%
}-B_{t_{1}^{m}},\ldots,B_{t_{2^{m}}^{m}}-B_{t_{2^{m}-1}^{m}}):\forall \phi \in
C_{b.Lip}(\mathbb{R}^{2^{m+1}d})\}.
\]
It is easy to check that $\mathcal{D}_{m}\subset \mathcal{D}_{n+1}$ and
$\mathcal{\bar{H}}_{m}\subset \mathcal{\bar{H}}_{m+1}$. Set $\mathcal{D}%
=\cup_{m=0}^{\infty}\mathcal{D}_{m}$ and $\mathcal{\bar{H}}=\cup_{m=0}%
^{\infty}\mathcal{\bar{H}}_{m}$.

Step 1: We construct a nonlinear expectation $\mathbb{\bar{E}}[\cdot]$ and a
sublinear expectation $\mathbb{E}^{\ast}[\cdot]$ on $(\bar{\Omega
},\mathcal{\bar{H}})$.

Let $\{ \mathbb{\tilde{F}}_{n_{i}}[\cdot]\}_{i=1}^{\infty}$ be any given
subsequence of $\{ \mathbb{\tilde{F}}_{n}[\cdot]\}_{n=1}^{\infty}$. Note that
$\{ \mathbb{\tilde{F}}_{n_{i}}[\cdot]\}_{i=1}^{\infty}$ and $\{ \mathbb{\hat
{F}}_{n_{i}}[\cdot]\}_{i=1}^{\infty}$ are tight. Then, by Theorem \ref{th-6},
there exists a subsequence $\{n_{i}^{0}\}_{i=1}^{\infty}$ of $\{n_{i}%
\}_{i=1}^{\infty}$ such that $\mathbb{\tilde{F}}_{n_{i}^{0}}[\phi]$ and
$\mathbb{\hat{F}}_{n_{i}^{0}}[\phi]$ converge as $i\rightarrow \infty$ for each
$\phi \in C_{b.Lip}(\mathbb{R}^{2d})$. Define a nonlinear expectation
$\mathbb{\bar{E}}_{0}[\cdot]:$ $\mathcal{\bar{H}}_{0}\rightarrow \mathbb{R}$
and a sublinear expectation $\mathbb{E}_{0}^{\ast}[\cdot]:$ $\mathcal{\bar{H}%
}_{0}\rightarrow \mathbb{R}$ as follows: for each $\phi \in C_{b.Lip}%
(\mathbb{R}^{2d})$,
\begin{equation}
\mathbb{\bar{E}}_{0}[\phi(B_{1}-B_{0})]:=\lim_{i\rightarrow \infty
}\mathbb{\tilde{F}}_{n_{i}^{0}}[\phi]\text{ and }\mathbb{E}_{0}^{\ast}%
[\phi(B_{1}-B_{0})]:=\lim_{i\rightarrow \infty}\mathbb{\hat{F}}_{n_{i}^{0}%
}[\phi]. \label{e2-07}%
\end{equation}
Set
\[
\xi_{n_{i}^{0}}=\left \{
\begin{array}
[c]{ll}%
(X_{n_{i}^{0},n_{i}^{0}},Y_{n_{i}^{0},n_{i}^{0}}),\text{ } & \text{if }%
n_{i}^{0}\text{ is an odd number,}\\
(0,0), & \text{if }n_{i}^{0}\text{ is an even number.}%
\end{array}
\right.
\]
For each given $\phi \in C_{b.Lip}(\mathbb{R}^{2d})$, we have%
\begin{align*}
|\mathbb{\tilde{E}}[\phi(\xi_{n_{i}^{0}})]-\mathbb{\tilde{E}}[\phi
(0)]|+|\mathbb{\hat{E}}[\phi(\xi_{n_{i}^{0}})]-\mathbb{\hat{E}}[\phi(0)]|  &
\leq2L_{\phi}\mathbb{\hat{E}}\left[  |X_{n_{i}^{0},n_{i}^{0}}|+|Y_{n_{i}%
^{0},n_{i}^{0}}|\right] \\
&  \leq2L_{\phi}\{(\mathbb{\hat{E}}[|X_{n_{i}^{0},1}|^{2}])^{1/2}%
+\mathbb{\hat{E}}[|Y_{n_{i}^{0},1}|]\},
\end{align*}
which tends to $0$ as $i\rightarrow \infty$ by (\ref{e2-7}). Thus $\xi
_{n_{i}^{0}}$ converges in distribution to $0$ under $\mathbb{\tilde{E}}%
[\cdot]$ and $\mathbb{\hat{E}}[\cdot]$ respectively. For $\phi \in
C_{b.Lip}(\mathbb{R}^{2d})$, define%
\[
\mathbb{\tilde{F}}_{n_{i}^{0}}^{1}[\phi]:=\mathbb{\tilde{E}}\left[
\phi \left(  \sum_{k=1}^{[n_{i}^{0}/2]}X_{n_{i}^{0},k},\sum_{k=1}^{[n_{i}%
^{0}/2]}Y_{n_{i}^{0},k}\right)  \right]  \text{, }\mathbb{\hat{F}}_{n_{i}^{0}%
}^{1}[\phi]:=\mathbb{\hat{E}}\left[  \phi \left(  \sum_{k=1}^{[n_{i}^{0}%
/2]}X_{n_{i}^{0},k},\sum_{k=1}^{[n_{i}^{0}/2]}Y_{n_{i}^{0},k}\right)  \right]
.
\]
It is easy to verify that $\{ \mathbb{\tilde{F}}_{n_{i}^{0}}^{1}%
\}_{i=1}^{\infty}$ and $\{ \mathbb{\hat{F}}_{n_{i}^{0}}^{1}\}_{i=1}^{\infty}$
are tight. Then, by Theorem \ref{th-6}, there exists a subsequence
$\{n_{i}^{1}\}_{i=1}^{\infty}$ of $\{n_{i}^{0}\}_{i=1}^{\infty}$ such that
$\mathbb{\tilde{F}}_{n_{i}^{1}}^{1}[\phi]$ and $\mathbb{\hat{F}}_{n_{i}^{1}%
}[\phi]$ converge as $i\rightarrow \infty$ for each $\phi \in C_{b.Lip}%
(\mathbb{R}^{2d})$. Define a nonlinear expectation $\mathbb{\bar{E}}_{1}%
[\cdot]:$ $\mathcal{\bar{H}}_{1}\rightarrow \mathbb{R}$ and a sublinear
expectation $\mathbb{E}_{1}^{\ast}[\cdot]:$ $\mathcal{\bar{H}}_{1}%
\rightarrow \mathbb{R}$ as follows: for each $\phi \in C_{b.Lip}(\mathbb{R}%
^{4d})$,%
\[
\mathbb{\bar{E}}_{1}[\phi(B_{1/2}-B_{0},B_{1}-B_{1/2})]:=\lim_{i\rightarrow
\infty}\mathbb{\tilde{F}}_{n_{i}^{1}}^{1}[\bar{\psi}],\text{ }\mathbb{E}%
_{1}^{\ast}[\phi(B_{1/2}-B_{0},B_{1}-B_{1/2})]:=\lim_{i\rightarrow \infty
}\mathbb{\hat{F}}_{n_{i}^{1}}[\psi],
\]
where $\bar{\psi}(x):=\lim_{i\rightarrow \infty}\mathbb{\tilde{F}}_{n_{i}^{1}%
}^{1}[\phi(x,\cdot)]$ and $\psi(x):=\lim_{i\rightarrow \infty}\mathbb{\hat{F}%
}_{n_{i}^{1}}[\phi(x,\cdot)]$ for $x\in \mathbb{R}^{2d}$. Repeating this
process, we can define a sequence of nonlinear expectations $\mathbb{\bar{E}%
}_{m}[\cdot]:$ $\mathcal{\bar{H}}_{m}\rightarrow \mathbb{R}$ and sublinear
expectations $\mathbb{E}_{m}^{\ast}[\cdot]:$ $\mathcal{\bar{H}}_{m}%
\rightarrow \mathbb{R}$, $m\geq0$, such that $B_{t_{j+1}^{m}}-B_{t_{j}^{m}}$ is
independent of $(B_{t_{1}^{m}}-B_{t_{0}^{m}}$,$\ldots$,$B_{t_{j}^{m}%
}-B_{t_{j-1}^{m}})$ for $j=1$,$\ldots$,$2^{m}-1$ under $\mathbb{\bar{E}}%
_{m}[\cdot]$ and $\mathbb{E}_{m}^{\ast}[\cdot]$,
\begin{equation}
\mathbb{\bar{E}}_{m}[\phi(B_{t_{j}^{m}}-B_{t_{j-1}^{m}})]=\lim_{i\rightarrow
\infty}\mathbb{\tilde{E}}\left[  \phi \left(  \sum_{k=1}^{[n_{i}^{m}/2^{m}%
]}X_{n_{i}^{m},k},\sum_{k=1}^{[n_{i}^{m}/2^{m}]}Y_{n_{i}^{m},k}\right)
\right]  , \label{e2-8}%
\end{equation}%
\begin{equation}
\mathbb{E}_{m}^{\ast}[\phi(B_{t_{j}^{m}}-B_{t_{j-1}^{m}})]=\lim_{i\rightarrow
\infty}\mathbb{\hat{E}}\left[  \phi \left(  \sum_{k=1}^{[n_{i}^{m}/2^{m}%
]}X_{n_{i}^{m},k},\sum_{k=1}^{[n_{i}^{m}/2^{m}]}Y_{n_{i}^{m},k}\right)
\right]  \label{e2-9}%
\end{equation}
for $\phi \in C_{b.Lip}(\mathbb{R}^{2d})$ and $j=1$,$\ldots$,$2^{m}$, where
$\{n_{i}^{m}\}_{i=1}^{\infty}$ is a subsequence of $\{n_{i}^{m-1}%
\}_{i=1}^{\infty}$ for $m\geq1$. By Proposition \ref{pro-7}, we obtain
$\mathbb{\bar{E}}_{m}|_{\mathcal{\bar{H}}_{m-1}}=\mathbb{\bar{E}}_{m-1}$ and
$\mathbb{E}_{m}^{\ast}|_{\mathcal{\bar{H}}_{m-1}}=\mathbb{E}_{m-1}^{\ast}$ for
$m\geq1$. Thus we can define a nonlinear expectation $\mathbb{\bar{E}}[\cdot]$
and a sublinear expectation $\mathbb{E}^{\ast}[\cdot]$ on $(\bar{\Omega
},\mathcal{\bar{H}})$ such that $\mathbb{\bar{E}}|_{\mathcal{\bar{H}}_{m}%
}=\mathbb{\bar{E}}_{m}$ and $\mathbb{E}^{\ast}|_{\mathcal{\bar{H}}_{m}%
}=\mathbb{E}_{m}^{\ast}$ for $m\geq0$. It is obvious that $\mathbb{\bar{E}%
}[\cdot]$ is dominated by $\mathbb{E}^{\ast}[\cdot]$.

Step 2: We study the properties of $\mathbb{\bar{E}}[\cdot]$ and
$\mathbb{E}^{\ast}[\cdot]$.

By (\ref{e2-4}), we obtain%
\begin{align*}
&  \mathbb{\hat{E}}\left[  \left(  |S_{n,k}^{X}|^{2}-\frac{kN}{n}\right)
^{+}+\left(  |S_{n,k}^{Y}|-\frac{kN}{n}\right)  ^{+}\right] \\
&  \leq \frac{k}{n}\mathbb{\hat{E}}\left[  \left(  n|X_{n,1}|^{2}-\frac{N}%
{2}\right)  ^{+}+(n|Y_{n,1}|-N)^{+}\right]  +\frac{4k\left(  n\mathbb{\hat{E}%
}\left[  |X_{n,1}|^{2}\right]  \right)  ^{2}}{nN},
\end{align*}
for $n\in \mathbb{N}$ and $k\leq n$. Thus, by (\ref{e2-00}), (\ref{e2-7}) and
(\ref{e2-9}), we obtain%
\begin{equation}%
\begin{array}
[c]{l}%
\sup_{m\geq0}2^{m}\mathbb{E}^{\ast}[(|\tilde{B}_{2^{-m}}-\tilde{B}_{0}%
|^{2}-N2^{-m})^{+}+(|\bar{B}_{2^{-m}}-\bar{B}_{0}|-N2^{-m})^{+}]\\
\leq \sup_{n\geq1}\{ \mathbb{\hat{E}}\left[  (n|X_{n,1}|^{2}-2^{-1}%
N)^{+}+(n|Y_{n,1}|-N)^{+}\right]  +4N^{-1}\left(  n\mathbb{\hat{E}}\left[
|X_{n,1}|^{2}\right]  \right)  ^{2}\} \\
\rightarrow0\text{ as }N\rightarrow \infty.
\end{array}
\label{e2-10}%
\end{equation}
Note that%
\begin{align*}
&  \sup_{m\geq0}2^{m}\mathbb{E}^{\ast}[|\tilde{B}_{2^{-m}}-\tilde{B}_{0}%
|^{2}+|\bar{B}_{2^{-m}}-\bar{B}_{0}|]\\
&  \leq2N+\sup_{m\geq0}2^{m}\mathbb{E}^{\ast}[(|\tilde{B}_{2^{-m}}-\tilde
{B}_{0}|^{2}-N2^{-m})^{+}+(|\bar{B}_{2^{-m}}-\bar{B}_{0}|-N2^{-m})^{+}].
\end{align*}
Then we get%
\begin{equation}
\sup_{m\geq0}2^{m}\mathbb{E}^{\ast}[|\tilde{B}_{2^{-m}}-\tilde{B}_{0}%
|^{2}+|\bar{B}_{2^{-m}}-\bar{B}_{0}|]\leq C \label{e2-11}%
\end{equation}
by (\ref{e2-10}), where $C>0$ is a constant. It follows from (\ref{e2-00}) and
(\ref{e2-9}) that%
\begin{equation}
\mathbb{E}^{\ast}[\tilde{B}_{2^{-m}}-\tilde{B}_{0}]=\mathbb{E}^{\ast}%
[-(\tilde{B}_{2^{-m}}-\tilde{B}_{0})]=0\text{ for }m\geq0. \label{e2-12}%
\end{equation}
For each $(A,p)\in \mathbb{S}_{d}\times \mathbb{R}^{d}$, by (\ref{e2-01}),
(\ref{e2-3}) and (\ref{e2-8}), we get%
\begin{equation}%
\begin{array}
[c]{l}%
\mathbb{\bar{E}}\left[  \frac{1}{2}\langle A(\tilde{B}_{2^{-m}}-\tilde{B}%
_{0}),\tilde{B}_{2^{-m}}-\tilde{B}_{0}\rangle+\langle p,\bar{B}_{2^{-m}}%
-\bar{B}_{0}\rangle \right] \\
=\lim_{i\rightarrow \infty}\frac{[n_{i}^{m}2^{-m}]}{n_{i}^{m}}n_{i}%
^{m}\mathbb{\tilde{E}}\left[  \frac{1}{2}\langle AX_{n_{i}^{m},1},X_{n_{i}%
^{m},1}\rangle+\langle p,Y_{n_{i}^{m},1}\rangle \right] \\
=\tilde{G}(A,p)2^{-m}.
\end{array}
\label{e2-13}%
\end{equation}

Step 3: Viscosity solution of PDE (\ref{e2-0}).

For each $\phi \in C_{b.Lip}(\mathbb{R}^{2d})$, we define%
\begin{equation}
u(t,x,y):=\mathbb{\bar{E}}[\phi(x+\tilde{B}_{t}-\tilde{B}_{0},y+\bar{B}%
_{t}-\bar{B}_{0})]\text{ for }t\in \mathcal{D}\text{, }x\text{, }y\in
\mathbb{R}^{d}. \label{e2-14}%
\end{equation}
It is clear that $u(0,x,y)=\phi(x,y)$. Since $\mathbb{\bar{E}}[\cdot]$ is
dominated by $\mathbb{E}^{\ast}[\cdot]$, we get $|u(t,x,y)-u(t,x^{\prime
},y^{\prime})|\leq L_{\phi}(|x-x^{\prime}|+|y-y^{\prime}|)$. Noting that
$B_{t}-B_{s}$ is independent of $B_{s}-B_{0}$ under $\mathbb{\bar{E}}[\cdot]$
for $s$, $t\in \mathcal{D}$ with $s<t$, we deduce%
\begin{equation}
u(t,x,y)=\mathbb{\bar{E}}[u(t-s,x+\tilde{B}_{s}-\tilde{B}_{0},y+\bar{B}%
_{s}-\bar{B}_{0})]. \label{e2-15}%
\end{equation}
Then, for $s$, $t\in \mathcal{D}$ with $s<t$, we have%
\begin{equation}%
\begin{array}
[c]{rl}%
|u(t,x,y)-u(t-s,x,y)| & \leq L_{\phi}\mathbb{E}^{\ast}[|\tilde{B}_{s}%
-\tilde{B}_{0}|+|\bar{B}_{s}-\bar{B}_{0}|]\\
& \leq L_{\phi}\{(\mathbb{E}^{\ast}[|\tilde{B}_{s}-\tilde{B}_{0}|^{2}%
])^{1/2}+\mathbb{E}^{\ast}[|\bar{B}_{s}-\bar{B}_{0}|]\} \\
& \leq L_{\phi}\{(Cs)^{1/2}+Cs\},
\end{array}
\label{e2-16}%
\end{equation}
where the last inequality is due to (\ref{e2-11}) and (\ref{e2-12}). It is
follows from (\ref{e2-16}) that $u$ can be continuously extended to a function
on $[0,1]\times \mathbb{R}^{2d}$. From (\ref{e2-15}) and (\ref{e2-16}), we
obtain%
\begin{equation}
u(t,x,y)=\mathbb{\bar{E}}[u(t-2^{-m},x+\tilde{B}_{2^{-m}}-\tilde{B}_{0}%
,y+\bar{B}_{2^{-m}}-\bar{B}_{0})] \label{e2-17}%
\end{equation}
for $t\in(0,1]$, $m\geq1$ with $2^{-m}<t$, $x$, $y\in \mathbb{R}^{d}$.

We show that $u$ is a viscosity subsolution of PDE (\ref{e2-0}). For each
fixed $(t,x,y)\in(0,1]\times \mathbb{R}^{d}\times \mathbb{R}^{d}$, let
$\varphi \in C_{b}^{2,3}([0,1]\times \mathbb{R}^{2d})$ satisfy $\varphi \geq u$
and $\varphi(t,x,y)=u(t,x,y)$. Then, by (\ref{e2-17}), we get%
\[
\varphi(t,x,y)\leq \mathbb{\bar{E}}[\varphi(t-2^{-m},x+\tilde{B}_{2^{-m}%
}-\tilde{B}_{0},y+\bar{B}_{2^{-m}}-\bar{B}_{0})]
\]
for $m\geq1$ with $2^{-m}<t$, which implies%
\begin{equation}
0\leq \mathbb{\bar{E}}[\varphi(t-2^{-m},x+\tilde{B}_{2^{-m}}-\tilde{B}%
_{0},y+\bar{B}_{2^{-m}}-\bar{B}_{0})-\varphi(t,x,y)]. \label{e2-18}%
\end{equation}
It follows from Taylor's expansion that%
\begin{equation}
\varphi(t-2^{-m},x+\tilde{B}_{2^{-m}}-\tilde{B}_{0},y+\bar{B}_{2^{-m}}-\bar
{B}_{0})-\varphi(t,x,y)=I_{m}+J_{m}^{1}+J_{m}^{2}+J_{m}^{3}, \label{e2-19}%
\end{equation}
where%
\[%
\begin{array}
[c]{rl}%
I_{m}= & -\partial_{t}\varphi(t,x,y)2^{-m}+\langle D_{x}\varphi(t,x,y),\tilde
{B}_{2^{-m}}-\tilde{B}_{0}\rangle+\langle D_{y}\varphi(t,x,y),\bar{B}_{2^{-m}%
}-\bar{B}_{0}\rangle \\
& +\frac{1}{2}\langle D_{x}^{2}\varphi(t,x,y)(\tilde{B}_{2^{-m}}-\tilde{B}%
_{0}),\tilde{B}_{2^{-m}}-\tilde{B}_{0}\rangle,
\end{array}
\]%
\[
J_{m}^{1}=2^{-m}\int_{0}^{1}[\partial_{t}\varphi(t,x,y)-\partial_{t}%
\varphi(t-\alpha2^{-m},x+\tilde{B}_{2^{-m}}-\tilde{B}_{0},y+\bar{B}_{2^{-m}%
}-\bar{B}_{0})]d\alpha,
\]%
\[
J_{m}^{2}=\int_{0}^{1}\langle D_{y}\varphi(t,x+\tilde{B}_{2^{-m}}-\tilde
{B}_{0},y+\alpha(\bar{B}_{2^{-m}}-\bar{B}_{0}))-D_{y}\varphi(t,x,y),\bar
{B}_{2^{-m}}-\bar{B}_{0}\rangle d\alpha,
\]%
\[
J_{m}^{3}=\int_{0}^{1}\langle \lbrack D_{x}^{2}\varphi(t,x+\alpha \beta
(\tilde{B}_{2^{-m}}-\tilde{B}_{0}),y)-D_{x}^{2}\varphi(t,x,y)](\tilde
{B}_{2^{-m}}-\tilde{B}_{0}),\tilde{B}_{2^{-m}}-\tilde{B}_{0}\rangle \alpha
d\beta d\alpha.
\]
By Proposition \ref{pro-1}, (\ref{e2-12}) and (\ref{e2-13}), we have%
\[
\mathbb{\bar{E}}[I_{m}]=-\partial_{t}\varphi(t,x,y)2^{-m}+\tilde{G}(D_{x}%
^{2}\varphi(t,x,y),D_{y}\varphi(t,x,y))2^{-m}.
\]
Since $|\mathbb{\bar{E}}[\varphi(t-2^{-m},x+\tilde{B}_{2^{-m}}-\tilde{B}%
_{0},y+\bar{B}_{2^{-m}}-\bar{B}_{0})-\varphi(t,x,y)]-\mathbb{\bar{E}}%
[I_{m}]|\leq \sum_{k=1}^{3}\mathbb{E}^{\ast}[|J_{m}^{k}|]$, we obtain%
\begin{equation}
\partial_{t}\varphi(t,x,y)-\tilde{G}(D_{x}^{2}\varphi(t,x,y),D_{y}%
\varphi(t,x,y))\leq \sum_{k=1}^{3}2^{m}\mathbb{E}^{\ast}[|J_{m}^{k}|]
\label{e2-20}%
\end{equation}
by (\ref{e2-18}). Noting that $\varphi \in C_{b}^{2,3}([0,1]\times
\mathbb{R}^{2d})$, it is easy to verify that%
\begin{align*}
2^{m}|J_{m}^{1}|  &  \leq|\partial_{t}\varphi(t,x,y)-\partial_{t}%
\varphi(t,x+\tilde{B}_{2^{-m}}-\tilde{B}_{0},y+\bar{B}_{2^{-m}}-\bar{B}%
_{0})|+||\partial_{tt}^{2}\varphi||_{\infty}2^{-m}\\
&  \leq||\partial_{tt}^{2}\varphi||_{\infty}2^{-m/4}+2||\partial_{t}%
\varphi||_{\infty}2^{m/4}(|\tilde{B}_{2^{-m}}-\tilde{B}_{0}|+|\bar{B}_{2^{-m}%
}-\bar{B}_{0}|)+||\partial_{tt}^{2}\varphi||_{\infty}2^{-m}.
\end{align*}
Then, by (\ref{e2-11}), we obtain%
\[
2^{m}\mathbb{E}^{\ast}[|J_{m}^{1}|]\leq||\partial_{tt}^{2}\varphi||_{\infty
}2^{-m/4}+2||\partial_{t}\varphi||_{\infty}2^{m/4}(\sqrt{C2^{-m}}%
+C2^{-m})+||\partial_{tt}^{2}\varphi||_{\infty}2^{-m},
\]
which implies $\lim_{m\rightarrow \infty}2^{m}\mathbb{E}^{\ast}[|J_{m}^{1}%
|]=0$. Noting that $|\bar{B}_{2^{-m}}-\bar{B}_{0}|\leq N2^{-m}+(|\bar
{B}_{2^{-m}}-\bar{B}_{0}|-N2^{-m})^{+}$ for each $N>0$, it is easy to verify
that%
\begin{align*}
|J_{m}^{2}|  &  \leq N2^{-m}(||D_{yx}^{2}\varphi||_{\infty}|\tilde{B}_{2^{-m}%
}-\tilde{B}_{0}|+||D_{yy}^{2}\varphi||_{\infty}|\bar{B}_{2^{-m}}-\bar{B}%
_{0}|)+2||D_{y}\varphi||_{\infty}(|\bar{B}_{2^{-m}}-\bar{B}_{0}|-N2^{-m}%
)^{+}\\
&  \leq C_{1}\{N2^{-m}(|\tilde{B}_{2^{-m}}-\tilde{B}_{0}|+|\bar{B}_{2^{-m}%
}-\bar{B}_{0}|)+(|\bar{B}_{2^{-m}}-\bar{B}_{0}|-N2^{-m})^{+}\},
\end{align*}
where $C_{1}>0$ is a constant independent of $m$ and $N$. It follows from
(\ref{e2-11}) that%
\[
\underset{m\rightarrow \infty}{\lim \sup}2^{m}\mathbb{E}^{\ast}[|J_{m}^{2}%
|]\leq \sup_{m\geq0}2^{m}C_{1}\mathbb{E}^{\ast}[(|\bar{B}_{2^{-m}}-\bar{B}%
_{0}|-N2^{-m})^{+}],
\]
which implies $\lim_{m\rightarrow \infty}2^{m}\mathbb{E}^{\ast}[|J_{m}^{2}|]=0$
by (\ref{e2-10}). Similarly, we can prove that $\lim_{m\rightarrow \infty}%
2^{m}\mathbb{E}^{\ast}[|J_{m}^{3}|]=0$. Thus we obtain%
\[
\partial_{t}\varphi(t,x,y)-\tilde{G}(D_{x}^{2}\varphi(t,x,y),D_{y}%
\varphi(t,x,y))\leq0
\]
by taking $m\rightarrow \infty$ in (\ref{e2-20}), which implies that $u$ is a
viscosity subsolution of PDE (\ref{e2-0}). Similarly, we can show that $u$ is
a viscosity supersolution of PDE (\ref{e2-0}). Thus $u$ is a viscosity
solution of PDE (\ref{e2-0}).

Step 4: Uniqueness of the limit.

We assert that (\ref{e2-1}) holds true. Otherwise, there exist a
$\varepsilon_{0}>0$ and a subsequence $\{n_{i}\}_{i=1}^{\infty}$ of
$\{n\}_{n=1}^{\infty}$ such that%
\begin{equation}
\left \vert \mathbb{\tilde{E}}\left[  \phi \left(  \sum_{k=1}^{n_{i}}X_{n_{i}%
,k},\sum_{k=1}^{n_{i}}Y_{n_{i},k}\right)  \right]  -u^{\phi}(1,0,0)\right \vert
\geq \varepsilon_{0}. \label{e2-21}%
\end{equation}
By Steps 1-3, there exists a subsequence $\{n_{i}^{0}\}_{i=1}^{\infty}$ of
$\{n_{i}\}_{i=1}^{\infty}$ such that%
\[
\lim_{i\rightarrow \infty}\mathbb{\tilde{E}}\left[  \phi \left(  \sum
_{k=1}^{n_{i}^{0}}X_{n_{i}^{0},k},\sum_{k=1}^{n_{i}^{0}}Y_{n_{i}^{0}%
,k}\right)  \right]  =u^{\phi}(1,0,0),
\]
which contradicts (\ref{e2-21}). Thus we obtain (\ref{e2-1}). $\Box$

\begin{remark}
\label{re-11}Generally speaking, we suppose that $X_{n,1}\overset{d}{=}%
\frac{1}{\sqrt{n}}X_{1,1}$ and $Y_{n,1}\overset{d}{=}\frac{1}{n}Y_{1,1}$ under
$\mathbb{\hat{E}}[\cdot]$ in the assumptions of Theorem \ref{th-8}. In this
situation, (\ref{e2-00}) is modified as follows:%
\begin{equation}
\mathbb{\hat{E}}[X_{1,1}]=\mathbb{\hat{E}}[-X_{1,1}]=0\text{ and }%
\lim_{N\rightarrow \infty}\mathbb{\hat{E}}\left[  (|X_{1,1}|^{2}-N)^{+}%
+(|Y_{1,1}|-N)^{+}\right]  =0. \label{e2-22}%
\end{equation}

\end{remark}

\begin{corollary}
\label{cor-1}Let a nonlinear expectation $\mathbb{\tilde{E}}[\cdot]$ be
dominated by a sublinear expectation $\mathbb{\hat{E}}[\cdot]$ on
$(\Omega,\mathcal{H})$, and let $\{(X_{n},Y_{n})\}_{n=1}^{\infty}$ be a
sequence of $2d$-dimensional random vectors such that $(X_{n+1},Y_{n+1}%
)\overset{d}{=}(X_{1},Y_{1})$ and $(X_{n+1},Y_{n+1})$ is independent of
$(X_{1}$,$\ldots$,$X_{n}$,$Y_{1}$,$\ldots$,$Y_{n})$ for $n\geq1$ under
$\mathbb{\tilde{E}}[\cdot]$ and $\mathbb{\hat{E}}[\cdot]$ respectively.
Suppose that%
\[
\mathbb{\hat{E}}[X_{1}]=\mathbb{\hat{E}}[-X_{1}]=0\text{ and }\lim
_{N\rightarrow \infty}\mathbb{\hat{E}}\left[  (|X_{1}|^{2}-N)^{+}%
+(|Y_{1}|-N)^{+}\right]  =0.
\]
Then, for each $\phi \in C_{b.Lip}(\mathbb{R}^{2d})$, we have%
\begin{equation}
\lim_{n\rightarrow \infty}\frac{1}{n}\mathbb{\tilde{E}}\left[  n\phi \left(
\frac{1}{\sqrt{n}}\sum_{i=1}^{n}X_{i},\frac{1}{n}\sum_{i=1}^{n}Y_{i}\right)
\right]  =u^{\phi}(1,0,0),\label{e2-23}%
\end{equation}
where $u^{\phi}$ is the unique viscosity solution of the following PDE:%
\begin{equation}
\partial_{t}u-\tilde{G}(D_{x}^{2}u,D_{y}u)=0\text{, }u(0,x,y)=\phi
(x,y),\label{e2-25}%
\end{equation}
and%
\begin{equation}
\tilde{G}(A,p):=\mathbb{\tilde{E}}\left[  \frac{1}{2}\langle AX_{1}%
,X_{1}\rangle+\langle p,Y_{1}\rangle \right]  \text{ for }(A,p)\in
\mathbb{S}_{d}\times \mathbb{R}^{d}.\label{e2-24}%
\end{equation}

\end{corollary}

\begin{proof}
By using the product space technique (see Proposition 1.3.17 in \cite{P2019}),
we can construct a sequence of $2d$-dimensional random vectors $\{(X_{n,i}%
,Y_{n,i}):n\in \mathbb{N}$, $i=1$,$\ldots$,$n\}$ under a nonlinear expectation
$\mathbb{\tilde{E}}_{1}[\cdot]$ and a sublinear expectation $\mathbb{\hat{E}%
}_{1}[\cdot]$ satisfying%
\[
\mathbb{\tilde{E}}_{1}[\Phi(X_{n,1},\ldots,X_{n,n},Y_{n,1},\ldots
,Y_{n,n})]=\frac{1}{n}\mathbb{\tilde{E}}\left[  n\Phi \left(  \frac{1}{\sqrt
{n}}X_{1},\ldots,\frac{1}{\sqrt{n}}X_{n},\frac{1}{n}Y_{1},\ldots,\frac{1}%
{n}Y_{n}\right)  \right]
\]
and%
\[
\mathbb{\hat{E}}_{1}[\Phi(X_{n,1},\ldots,X_{n,n},Y_{n,1},\ldots,Y_{n,n}%
)]=\mathbb{\hat{E}}\left[  \Phi \left(  \frac{1}{\sqrt{n}}X_{1},\ldots,\frac
{1}{\sqrt{n}}X_{n},\frac{1}{n}Y_{1},\ldots,\frac{1}{n}Y_{n}\right)  \right]
\]
for each $\Phi \in C_{b.Lip}(\mathbb{R}^{2nd})$. It is easy to check that
$\mathbb{\tilde{E}}_{1}[\cdot]$ is dominated by $\mathbb{\hat{E}}_{1}[\cdot]$,
$(X_{n,i},Y_{n,i})\overset{d}{=}(X_{n,1},Y_{n,1})$ and $(X_{n,i},Y_{n,i})$ is
independent of $(X_{n,1}$,$\ldots$,$X_{n,i-1}$,$Y_{n,1}$,$\ldots$,$Y_{n,i-1})$
for $n\in \mathbb{N}$, $i=2$,$\ldots$,$n$ under $\mathbb{\tilde{E}}_{1}[\cdot]$
and $\mathbb{\hat{E}}_{1}[\cdot]$ respectively. Since $X_{n,1}\overset{d}%
{=}\frac{1}{\sqrt{n}}X_{1,1}$ under $\mathbb{\hat{E}}_{1}[\cdot]$, we obtain
(\ref{e2-22}). For each given $(A,p)\in \mathbb{S}_{d}\times \mathbb{R}^{d}$,
\[
n\mathbb{\tilde{E}}_{1}\left[  \frac{1}{2}\langle AX_{n,1},X_{n,1}%
\rangle+\langle p,Y_{n,1}\rangle \right]  =\mathbb{\tilde{E}}\left[  \frac
{1}{2}\langle AX_{1},X_{1}\rangle+\langle p,Y_{1}\rangle \right]  ,
\]
which implies (\ref{e2-24}). Note that%
\[
\mathbb{\tilde{E}}_{1}\left[  \phi \left(  \sum_{i=1}^{n}X_{n,i},\sum_{i=1}%
^{n}Y_{n,i}\right)  \right]  =\frac{1}{n}\mathbb{\tilde{E}}\left[
n\phi \left(  \frac{1}{\sqrt{n}}\sum_{i=1}^{n}X_{i},\frac{1}{n}\sum_{i=1}%
^{n}Y_{i}\right)  \right]  .
\]
Then we obtain (\ref{e2-23}) by Theorem \ref{th-8}.
\end{proof}

\begin{remark}
\label{re-new-1}When $\mathbb{\tilde{E}}[\cdot]$ is a convex expectation,
(\ref{e2-23}) was first obtained in \cite{BK}.
\end{remark}

Now we give an example to apply Theorem \ref{th-8}. We only need to provide
the distribution of $(X_{n,1},Y_{n,1})$ under $\mathbb{\tilde{E}}[\cdot]$ and
$\mathbb{\hat{E}}[\cdot]$ respectively. Because the existence of
$\{(X_{n,i},Y_{n,i}):n\in \mathbb{N}$, $i=1$,$\ldots$,$n\}$ can be obtained by
constructing the product space.

\begin{example}
\label{ex-1}Let $\{X_{n,i}:n\in \mathbb{N}$, $i=1$,$\ldots$,$n\}$ be a sequence
of $d$-dimensional random vectors and let $\mathcal{P}$ be a given weakly
compact and convex set of probability measures on $(\mathbb{R}^{d}%
,\mathcal{B}(\mathbb{R}^{d}))$ satisfying%
\[
\int_{\mathbb{R}^{d}}xdP=0\text{ for each }P\in \mathcal{P}\text{ and }%
\lim_{N\rightarrow \infty}\sup_{P\in \mathcal{P}}\int_{\mathbb{R}^{d}}%
(|x|^{2}-N)^{+}dP=0.
\]
Define%
\[
\mathbb{\hat{E}}[\phi(X_{n,1})]=\sup_{P\in \mathcal{P}}\int_{\mathbb{R}^{d}%
}\phi(x/\sqrt{n})dP\text{ for }n\geq1\text{ and }\phi \in C_{b.Lip}%
(\mathbb{R}^{d}).
\]
It is easy to verify that $X_{n,1}\overset{d}{=}\frac{1}{\sqrt{n}}X_{1,1}$
under $\mathbb{\hat{E}}[\cdot]$ and (\ref{e2-22}) holds true. Let
$\rho:\mathcal{P}\rightarrow \lbrack0,\infty)$ be a given convex function such
that $\{P\in \mathcal{P}:\rho(P)=0\}$ is nonempty. Define%
\[
\mathbb{\tilde{E}}[\phi(X_{n,1})]=\sup_{P\in \mathcal{P}}\left \{
\int_{\mathbb{R}^{d}}\phi(x/\sqrt{n})dP-\frac{1}{n}\rho(P)\right \}  \text{ for
}n\geq1\text{ and }\phi \in C_{b.Lip}(\mathbb{R}^{d}).
\]
It is easy to check that $\mathbb{\tilde{E}}[\cdot]$ is a convex expectation
and $\mathbb{\tilde{E}}[\cdot]$ is dominated by $\mathbb{\hat{E}}[\cdot]$. For
each $A\in \mathbb{S}_{d}$,
\[
n\mathbb{\tilde{E}}\left[  \frac{1}{2}\langle AX_{n,1},X_{n,1}\rangle \right]
=\sup_{P\in \mathcal{P}}\left \{  \int_{\mathbb{R}^{d}}\frac{1}{2}\langle
Ax,x\rangle dP-\rho(P)\right \}  ,
\]
which implies%
\[
\tilde{G}(A)=\sup_{P\in \mathcal{P}}\left \{  \int_{\mathbb{R}^{d}}\frac{1}%
{2}\langle Ax,x\rangle dP-\rho(P)\right \}  \text{ for }A\in \mathbb{S}_{d}.
\]
Thus, by Theorem \ref{th-8}, we obtain%
\[
\lim_{n\rightarrow \infty}\mathbb{\tilde{E}}\left[  \phi \left(  \sum_{i=1}%
^{n}X_{n,i}\right)  \right]  =\tilde{u}^{\phi}(1,0)\text{ for }\phi \in
C_{b.Lip}(\mathbb{R}^{d}),
\]
where $\tilde{u}^{\phi}$ is the unique viscosity solution of the following
PDE:%
\[
\partial_{t}u-\tilde{G}(D^{2}u)=0\text{, }u(0,x)=\phi(x).
\]

Let $\{ \rho_{\alpha}:\alpha \in I\}$ be a given family of nonnegative convex
functions on $\mathcal{P}$ such that $\{P\in \mathcal{P}:\rho_{\alpha}(P)=0\}$
is nonempty for each $\alpha \in I$. Define%
\[
\mathbb{\bar{E}}[\phi(X_{n,1})]=\inf_{\alpha \in I}\sup_{P\in \mathcal{P}%
}\left \{  \int_{\mathbb{R}^{d}}\phi(x/\sqrt{n})dP-\frac{1}{n}\rho_{\alpha
}(P)\right \}  \text{ for }n\geq1\text{ and }\phi \in C_{b.Lip}(\mathbb{R}%
^{d}).
\]
It is easy to check that $\mathbb{\bar{E}}[\cdot]$ is a nonlinear expectation
and $\mathbb{\bar{E}}[\cdot]$ is dominated by $\mathbb{\hat{E}}[\cdot]$. For
each $A\in \mathbb{S}_{d}$,
\[
n\mathbb{\bar{E}}\left[  \frac{1}{2}\langle AX_{n,1},X_{n,1}\rangle \right]
=\inf_{\alpha \in I}\sup_{P\in \mathcal{P}}\left \{  \int_{\mathbb{R}^{d}}%
\frac{1}{2}\langle Ax,x\rangle dP-\rho_{\alpha}(P)\right \}  ,
\]
which implies%
\[
\bar{G}(A)=\inf_{\alpha \in I}\sup_{P\in \mathcal{P}}\left \{  \int
_{\mathbb{R}^{d}}\frac{1}{2}\langle Ax,x\rangle dP-\rho_{\alpha}(P)\right \}
\text{ for }A\in \mathbb{S}_{d}.
\]
Thus, by Theorem \ref{th-8}, we obtain%
\[
\lim_{n\rightarrow \infty}\mathbb{\bar{E}}\left[  \phi \left(  \sum_{i=1}%
^{n}X_{n,i}\right)  \right]  =\bar{u}^{\phi}(1,0)\text{ for }\phi \in
C_{b.Lip}(\mathbb{R}^{d}),
\]
where $\bar{u}^{\phi}$ is the unique viscosity solution of the following PDE:%
\[
\partial_{t}u-\bar{G}(D^{2}u)=0\text{, }u(0,x)=\phi(x).
\]

\end{example}

\section{A special case}

Let $\{(X_{n},Y_{n})\}_{n=1}^{\infty}$ be a sequence of $2d$-dimensional
random vectors. In this section, we consider the following special case:%
\begin{equation}
X_{n,i}=\frac{1}{\sqrt{n}}X_{i}\text{ and }Y_{n,i}=\frac{1}{n}Y_{i}\text{ for
}i=1,\ldots,n. \label{e3-0}%
\end{equation}
According to Theorem \ref{th-8}, we have the following theorem.

\begin{theorem}
\label{th-12}Let a nonlinear expectation $\mathbb{\tilde{E}}[\cdot]$ be
dominated by a sublinear expectation $\mathbb{\hat{E}}[\cdot]$ on
$(\Omega,\mathcal{H})$, and let $\{(X_{n},Y_{n})\}_{n=1}^{\infty}$ be a
sequence of $2d$-dimensional random vectors such that $(X_{n+1},Y_{n+1}%
)\overset{d}{=}(X_{1},Y_{1})$ and $(X_{n+1},Y_{n+1})$ is independent of
$(X_{1}$,$\ldots$,$X_{n}$,$Y_{1}$,$\ldots$,$Y_{n})$ for $n\geq1$ under
$\mathbb{\tilde{E}}[\cdot]$ and $\mathbb{\hat{E}}[\cdot]$ respectively.
Suppose that%
\begin{equation}
\mathbb{\hat{E}}[X_{1}]=\mathbb{\hat{E}}[-X_{1}]=0\text{ and }\lim
_{N\rightarrow \infty}\mathbb{\hat{E}}\left[  (|X_{1}|^{2}-N)^{+}%
+(|Y_{1}|-N)^{+}\right]  =0, \label{e3-1}%
\end{equation}
and for each given $(A,p)\in \mathbb{S}_{d}\times \mathbb{R}^{d}$,%
\begin{equation}
\tilde{G}(A,p):=\lim_{n\rightarrow \infty}n\mathbb{\tilde{E}}\left[  \frac
{1}{n}\left(  \frac{1}{2}\langle AX_{1},X_{1}\rangle+\langle p,Y_{1}%
\rangle \right)  \right]  \text{ exists.} \label{e3-2}%
\end{equation}
Then for each $\phi \in C_{b.Lip}(\mathbb{R}^{2d})$, we have%
\begin{equation}
\lim_{n\rightarrow \infty}\mathbb{\tilde{E}}\left[  \phi \left(  \frac{1}%
{\sqrt{n}}\sum_{i=1}^{n}X_{i},\frac{1}{n}\sum_{i=1}^{n}Y_{i}\right)  \right]
=u^{\phi}(1,0,0), \label{e3-3}%
\end{equation}
where $u^{\phi}$ is the unique viscosity solution of the following PDE:%
\begin{equation}
\partial_{t}u-\tilde{G}(D_{x}^{2}u,D_{y}u)=0\text{, }u(0,x,y)=\phi(x,y).
\label{e3-4}%
\end{equation}

\end{theorem}

\begin{proof}
Set $(X_{n,i},Y_{n,i})$ as in (\ref{e3-0}) for $i=1,\ldots,n$. It is easy to
verify that $\{(X_{n,i},Y_{n,i}):n\in \mathbb{N}$, $i=1$,$\ldots$,$n\}$
satisfies the assumptions in Theorem \ref{th-8}. Thus we obtain (\ref{e3-3})
by Theorem \ref{th-8}.
\end{proof}

The following proposition shows that (\ref{e3-2}) holds true if
$\mathbb{\tilde{E}}[\cdot]$ is a convex expectation.

\begin{proposition}
\label{pro-13}Let a convex expectation $\mathbb{\tilde{E}}[\cdot]$ be
dominated by a sublinear expectation $\mathbb{\hat{E}}[\cdot]$ on
$(\Omega,\mathcal{H})$, and let $(X_{1},Y_{1})$ be a $2d$-dimensional random
vector satisfying%
\begin{equation}
\lim_{N\rightarrow \infty}\mathbb{\hat{E}}\left[  (|X_{1}|^{2}-N)^{+}%
+(|Y_{1}|-N)^{+}\right]  =0. \label{e3-5}%
\end{equation}
Then $\lim_{n\rightarrow \infty}n\mathbb{\tilde{E}}\left[  \frac{1}{n}(\frac
{1}{2}\langle AX_{1},X_{1}\rangle+\langle p,Y_{1}\rangle)\right]  $ exists for
each $(A,p)\in \mathbb{S}_{d}\times \mathbb{R}^{d}$.
\end{proposition}

\begin{proof}
For each fixed $(A,p)\in \mathbb{S}_{d}\times \mathbb{R}^{d}$, we have $\langle
AX_{1},X_{1}\rangle+\langle p,Y_{1}\rangle \in L^{1}(\Omega)$ by (\ref{e3-5})
and define%
\[
f(t)=\mathbb{\tilde{E}}\left[  t\left(  \frac{1}{2}\langle AX_{1},X_{1}%
\rangle+\langle p,Y_{1}\rangle \right)  \right]  \text{ for }t\in
\mathbb{R}\text{.}%
\]
Since $\mathbb{\tilde{E}}[\cdot]$ is a convex expectation, it is easy to check
that $f(\cdot)$ is a convex function on $\mathbb{R}$. Thus
\[
f_{+}^{\prime}(0)=\lim_{\delta \downarrow0}\frac{f(\delta)-f(0)}{\delta}%
=\lim_{\delta \downarrow0}\frac{f(\delta)}{\delta}%
\]
exists, which implies $\lim_{n\rightarrow \infty}n\mathbb{\tilde{E}}\left[
\frac{1}{n}(\frac{1}{2}\langle AX_{1},X_{1}\rangle+\langle p,Y_{1}%
\rangle)\right]  =f_{+}^{\prime}(0)$.
\end{proof}

The following proposition shows that $\tilde{G}(\cdot)$ defined in
(\ref{e3-2}) satisfies the positive homogeneity.

\begin{proposition}
\label{pro-14}Let a nonlinear expectation $\mathbb{\tilde{E}}[\cdot]$ be
dominated by a sublinear expectation $\mathbb{\hat{E}}[\cdot]$ on
$(\Omega,\mathcal{H})$, and let $(X_{1},Y_{1})$ be a $2d$-dimensional random
vector satisfying (\ref{e3-5}). Suppose that (\ref{e3-2}) holds true for each
$(A,p)\in \mathbb{S}_{d}\times \mathbb{R}^{d}$. Then, for each $(A,p)\in
\mathbb{S}_{d}\times \mathbb{R}^{d}$,
\begin{equation}
\tilde{G}(\lambda A,\lambda p)=\lambda \tilde{G}(A,p)\text{ for }\lambda \geq0.
\label{e3-6}%
\end{equation}

\end{proposition}

\begin{proof}
Since $\mathbb{\tilde{E}}[\cdot]$ is dominated by $\mathbb{\hat{E}}[\cdot]$,
we obtain that, for each $(A,p)\in \mathbb{S}_{d}\times \mathbb{R}^{d}$,
$\lambda_{1}$, $\lambda_{2}\geq0$,
\begin{align*}
&  \left \vert n\mathbb{\tilde{E}}\left[  \frac{1}{n}\left(  \frac{1}{2}%
\langle \lambda_{1}AX_{1},X_{1}\rangle+\langle \lambda_{1}p,Y_{1}\rangle \right)
\right]  -n\mathbb{\tilde{E}}\left[  \frac{1}{n}\left(  \frac{1}{2}%
\langle \lambda_{2}AX_{1},X_{1}\rangle+\langle \lambda_{2}p,Y_{1}\rangle \right)
\right]  \right \vert \\
&  \leq \mathbb{\hat{E}}\left[  \left \vert \frac{1}{2}\langle(\lambda
_{1}-\lambda_{2})AX_{1},X_{1}\rangle+\langle(\lambda_{1}-\lambda_{2}%
)p,Y_{1}\rangle \right \vert \right] \\
&  =|\lambda_{1}-\lambda_{2}|\mathbb{\hat{E}}\left[  \left \vert \frac{1}%
{2}\langle AX_{1},X_{1}\rangle+\langle p,Y_{1}\rangle \right \vert \right]  ,
\end{align*}
which implies%
\begin{equation}
|\tilde{G}(\lambda_{1}A,\lambda_{1}p)-\tilde{G}(\lambda_{2}A,\lambda
_{2}p)|\leq|\lambda_{1}-\lambda_{2}|\mathbb{\hat{E}}\left[  \left \vert
\frac{1}{2}\langle AX_{1},X_{1}\rangle+\langle p,Y_{1}\rangle \right \vert
\right]  . \label{e3-7}%
\end{equation}
For each given positive rational number $\lambda$, there exist two positive
integers $n_{1}$ and $n_{2}$ such that $\lambda=n_{1}/n_{2}$. Then%
\begin{align*}
\tilde{G}(\lambda A,\lambda p)  &  =\lim_{i\rightarrow \infty}n_{1}%
i\mathbb{\tilde{E}}\left[  \frac{1}{n_{1}i}\left(  \frac{1}{2}\langle \lambda
AX_{1},X_{1}\rangle+\langle \lambda p,Y_{1}\rangle \right)  \right] \\
&  =\lim_{i\rightarrow \infty}\lambda n_{2}i\mathbb{\tilde{E}}\left[  \frac
{1}{n_{2}i}\left(  \frac{1}{2}\langle \lambda AX_{1},X_{1}\rangle
+\langle \lambda p,Y_{1}\rangle \right)  \right] \\
&  =\lambda \tilde{G}(A,p)
\end{align*}
for each $(A,p)\in \mathbb{S}_{d}\times \mathbb{R}^{d}$. By (\ref{e3-7}), we
obtain (\ref{e3-6}).
\end{proof}

In the following, we give some examples.

\begin{example}
\label{ex-2}Let $\{X_{n}\}_{n=1}^{\infty}$ be a sequence of $d$-dimensional
random vectors. Similar to Example \ref{ex-1}, we only need to provide the
distribution of $X_{1}$. Let $\mathcal{P}$ be a given weakly compact and
convex set of probability measures on $(\mathbb{R}^{d},\mathcal{B}%
(\mathbb{R}^{d}))$ satisfying%
\[
\int_{\mathbb{R}^{d}}xdP=0\text{ for each }P\in \mathcal{P}\text{ and }%
\lim_{N\rightarrow \infty}\sup_{P\in \mathcal{P}}\int_{\mathbb{R}^{d}}%
(|x|^{2}-N)^{+}dP=0.
\]
Define%
\[
\mathbb{\hat{E}}[\phi(X_{1})]=\sup_{P\in \mathcal{P}}\int_{\mathbb{R}^{d}}%
\phi(x)dP\text{ for }\phi \in C_{b.Lip}(\mathbb{R}^{d}).
\]
It is easy to verify that (\ref{e3-1}) holds true. Let $\rho:\mathcal{P}%
\rightarrow \lbrack0,\infty)$ be a given convex function such that
$\{P\in \mathcal{P}:\rho(P)=0\}$ is nonempty. Define%
\[
\mathbb{\tilde{E}}[\phi(X_{1})]=\sup_{P\in \mathcal{P}}\left \{  \int
_{\mathbb{R}^{d}}\phi(x)dP-\rho(P)\right \}  \text{ for }\phi \in C_{b.Lip}%
(\mathbb{R}^{d}).
\]
It is easy to check that $\mathbb{\tilde{E}}[\cdot]$ is a convex expectation
and $\mathbb{\tilde{E}}[\cdot]$ is dominated by $\mathbb{\hat{E}}[\cdot]$.
Define%
\begin{equation}
\rho^{\ast}(P)=\underset{Q\rightarrow P}{\lim \inf}\rho(Q)\text{ for }%
P\in \mathcal{P}\text{,} \label{e3-08}%
\end{equation}
where $Q\rightarrow P$ is considered in the sense of weak convergence, and $Q$
can be equal to $P$. It is easy to verify that $\rho^{\ast}:\mathcal{P}%
\rightarrow \lbrack0,\infty)$ satisfies the following properties:

\begin{description}
\item[(i)] $\rho^{\ast}(P)\leq \rho(P)$ for each $P\in \mathcal{P}$.

\item[(ii)] For each given $P\in \mathcal{P}$, there exists $\{Q_{i}%
\}_{i=1}^{\infty}\subset \mathcal{P}$ such that $\{Q_{i}\}_{i=1}^{\infty}$
converges weakly to $P$ and $\rho(Q_{i})\rightarrow \rho^{\ast}(P)$ as
$i\rightarrow \infty$.

\item[(iii)] If $\{P_{i}\}_{i=1}^{\infty}\subset \mathcal{P}$ converges weakly
to $P$, then
\begin{equation}
\rho^{\ast}(P)\leq \underset{i\rightarrow \infty}{\lim \inf}\rho^{\ast}(P_{i}).
\label{e3-8}%
\end{equation}

\end{description}

Now we prove that%
\begin{equation}
\mathbb{\tilde{E}}[\phi(X_{1})]=\sup_{P\in \mathcal{P}}\left \{  \int
_{\mathbb{R}^{d}}\phi(x)dP-\rho^{\ast}(P)\right \}  \text{ for }\phi \in
C_{b.Lip}(\mathbb{R}^{d}). \label{e3-9}%
\end{equation}
Since $\rho^{\ast}(\cdot)\leq \rho(\cdot)$, we have $\mathbb{\tilde{E}}%
[\phi(X_{1})]\leq \sup_{P\in \mathcal{P}}\{ \int_{\mathbb{R}^{d}}\phi
(x)dP-\rho^{\ast}(P)\}$. For each given $P\in \mathcal{P}$, we can find
$\{Q_{i}\}_{i=1}^{\infty}$ satisfying (ii). Then%
\[
\int_{\mathbb{R}^{d}}\phi(x)dP-\rho^{\ast}(P)=\lim_{i\rightarrow \infty}\left(
\int_{\mathbb{R}^{d}}\phi(x)dQ_{i}-\rho(Q_{i})\right)  \leq \mathbb{\tilde{E}%
}[\phi(X_{1})],
\]
which implies $\sup_{P\in \mathcal{P}}\{ \int_{\mathbb{R}^{d}}\phi
(x)dP-\rho^{\ast}(P)\} \leq \mathbb{\tilde{E}}[\phi(X_{1})]$. Thus we obtain
(\ref{e3-9}). It is clear that%
\[
n\mathbb{\tilde{E}}\left[  \frac{1}{n}\left(  \frac{1}{2}\langle AX_{1}%
,X_{1}\rangle \right)  \right]  =\sup_{P\in \mathcal{P}}\left \{  \int
_{\mathbb{R}^{d}}\frac{1}{2}\langle Ax,x\rangle dP-n\rho^{\ast}(P)\right \}  .
\]
Set $\mathcal{P}^{\ast}:=\{P\in \mathcal{P}:\rho^{\ast}(P)=0\}$ and
$\mathcal{P}_{k}:=\{P\in \mathcal{P}:\rho^{\ast}(P)\leq1/k\}$ for
$k\in \mathbb{N}$. It follows from (\ref{e3-8}) that $\mathcal{P}^{\ast}$ and
$\mathcal{P}_{k}$, $k\in \mathbb{N}$, are weakly compact. In the following, we
show that%
\begin{equation}
\lim_{n\rightarrow \infty}\sup_{P\in \mathcal{P}}\left \{  \int_{\mathbb{R}^{d}%
}\frac{1}{2}\langle Ax,x\rangle dP-n\rho^{\ast}(P)\right \}  =\sup
_{P\in \mathcal{P}^{\ast}}\int_{\mathbb{R}^{d}}\frac{1}{2}\langle Ax,x\rangle
dP. \label{e3-10}%
\end{equation}
It is obvious that
\begin{equation}
\sup_{P\in \mathcal{P}}\left \{  \int_{\mathbb{R}^{d}}\frac{1}{2}\langle
Ax,x\rangle dP-n\rho^{\ast}(P)\right \}  \geq \sup_{P\in \mathcal{P}^{\ast}}%
\int_{\mathbb{R}^{d}}\frac{1}{2}\langle Ax,x\rangle dP \label{e3-11}%
\end{equation}
for $n\geq1$. Since%
\begin{align*}
&  \sup_{P\in \mathcal{P}}\left \{  \int_{\mathbb{R}^{d}}\frac{1}{2}\langle
Ax,x\rangle dP-n\rho^{\ast}(P)\right \} \\
&  \leq \max \left \{  \sup_{P\in \mathcal{P}_{k}}\int_{\mathbb{R}^{d}}\frac{1}%
{2}\langle Ax,x\rangle dP,\sup_{P\in \mathcal{P}\backslash \mathcal{P}_{k}%
}\left \{  \int_{\mathbb{R}^{d}}\frac{1}{2}\langle Ax,x\rangle dP-\frac{n}%
{k}\right \}  \right \}
\end{align*}
for each $k\in \mathbb{N}$, we obtain%
\begin{equation}
\lim_{n\rightarrow \infty}\sup_{P\in \mathcal{P}}\left \{  \int_{\mathbb{R}^{d}%
}\frac{1}{2}\langle Ax,x\rangle dP-n\rho^{\ast}(P)\right \}  \leq \sup
_{P\in \mathcal{P}_{k}}\int_{\mathbb{R}^{d}}\frac{1}{2}\langle Ax,x\rangle dP
\label{e3-12}%
\end{equation}
for each $k\in \mathbb{N}$. Note that $\sup_{Q\in \mathcal{P}}\int
_{\mathbb{R}^{d}}(|x|^{2}-N)^{+}dQ\rightarrow0$ as $N\rightarrow \infty$. Then
we obtain $\int_{\mathbb{R}^{d}}\frac{1}{2}\langle Ax,x\rangle dQ_{i}%
\rightarrow \int_{\mathbb{R}^{d}}\frac{1}{2}\langle Ax,x\rangle dQ$ by Lemma 29
in \cite{DHP11} when $\{Q_{i}\}_{i=1}^{\infty}$ converges weakly to $Q$. Thus
we can find $P_{k}\in \mathcal{P}_{k}$ such that%
\begin{equation}
\sup_{P\in \mathcal{P}_{k}}\int_{\mathbb{R}^{d}}\frac{1}{2}\langle Ax,x\rangle
dP=\int_{\mathbb{R}^{d}}\frac{1}{2}\langle Ax,x\rangle dP_{k}. \label{e3-13}%
\end{equation}
Since $\mathcal{P}$ is weakly compact, we can find a subsequence $\{P_{k_{i}%
}\}_{i=1}^{\infty}$ such that $\{P_{k_{i}}\}_{i=1}^{\infty}$ converges weakly
to $P^{\ast}$. By (\ref{e3-8}), it is easy to get $P^{\ast}\in \mathcal{P}%
^{\ast}$. Then, by (\ref{e3-12}) and (\ref{e3-13}), we deduce
\begin{equation}
\lim_{n\rightarrow \infty}\sup_{P\in \mathcal{P}}\left \{  \int_{\mathbb{R}^{d}%
}\frac{1}{2}\langle Ax,x\rangle dP-n\rho^{\ast}(P)\right \}  \leq
\lim_{i\rightarrow \infty}\int_{\mathbb{R}^{d}}\frac{1}{2}\langle Ax,x\rangle
dP_{k_{i}}=\int_{\mathbb{R}^{d}}\frac{1}{2}\langle Ax,x\rangle dP^{\ast}.
\label{e3-14}%
\end{equation}
Thus we obtain (\ref{e3-10}) by (\ref{e3-11}) and (\ref{e3-14}), which implies%
\[
\tilde{G}(A)=\sup_{P\in \mathcal{P}^{\ast}}\int_{\mathbb{R}^{d}}\frac{1}%
{2}\langle Ax,x\rangle dP\text{ for }A\in \mathbb{S}_{d}.
\]
Thus, by Theorem \ref{th-12}, we obtain%
\[
\lim_{n\rightarrow \infty}\mathbb{\tilde{E}}\left[  \phi \left(  \frac{1}%
{\sqrt{n}}\sum_{i=1}^{n}X_{i}\right)  \right]  =\tilde{u}^{\phi}(1,0)\text{
for }\phi \in C_{b.Lip}(\mathbb{R}^{d}),
\]
where $\tilde{u}^{\phi}$ is the unique viscosity solution of the following
PDE:%
\[
\partial_{t}u-\tilde{G}(D^{2}u)=0\text{, }u(0,x)=\phi(x).
\]

Let $\{ \rho_{\alpha}:\alpha \in I\}$ be a given family of nonnegative convex
functions on $\mathcal{P}$ such that $\{P\in \mathcal{P}:\rho_{\alpha}(P)=0\}$
is nonempty for each $\alpha \in I$. Define%
\[
\mathbb{\bar{E}}[\phi(X_{1})]=\inf_{\alpha \in I}\sup_{P\in \mathcal{P}}\left \{
\int_{\mathbb{R}^{d}}\phi(x)dP-\rho_{\alpha}(P)\right \}  \text{ for }\phi \in
C_{b.Lip}(\mathbb{R}^{d}).
\]
It is easy to check that $\mathbb{\bar{E}}[\cdot]$ is a nonlinear expectation
and $\mathbb{\bar{E}}[\cdot]$ is dominated by $\mathbb{\hat{E}}[\cdot]$. For
each $A\in \mathbb{S}_{d}$,
\[
n\mathbb{\bar{E}}\left[  \frac{1}{n}\left(  \frac{1}{2}\langle AX_{1}%
,X_{1}\rangle \right)  \right]  =\inf_{\alpha \in I}\sup_{P\in \mathcal{P}%
}\left \{  \int_{\mathbb{R}^{d}}\frac{1}{2}\langle Ax,x\rangle dP-n\rho
_{\alpha}^{\ast}(P)\right \}  ,
\]
where the definition of $\rho_{\alpha}^{\ast}(P)$ is similar to that of
$\rho^{\ast}(P)$ in (\ref{e3-08}). Set $\mathcal{P}_{\alpha}^{\ast}%
:=\{P\in \mathcal{P}:\rho_{\alpha}^{\ast}(P)=0\}$ for each $\alpha \in I$. Then
we have
\[
n\mathbb{\bar{E}}\left[  \frac{1}{n}\left(  \frac{1}{2}\langle AX_{1}%
,X_{1}\rangle \right)  \right]  \geq \inf_{\alpha \in I}\sup_{P\in \mathcal{P}%
_{\alpha}^{\ast}}\int_{\mathbb{R}^{d}}\frac{1}{2}\langle Ax,x\rangle dP.
\]
By (\ref{e3-10}), we obtain%
\begin{align*}
\underset{n\rightarrow \infty}{\lim \sup}n\mathbb{\bar{E}}\left[  \frac{1}%
{n}\left(  \frac{1}{2}\langle AX_{1},X_{1}\rangle \right)  \right]   &
\leq \lim_{n\rightarrow \infty}\sup_{P\in \mathcal{P}}\left \{  \int
_{\mathbb{R}^{d}}\frac{1}{2}\langle Ax,x\rangle dP-n\rho_{\alpha}^{\ast
}(P)\right \} \\
&  =\sup_{P\in \mathcal{P}_{\alpha}^{\ast}}\int_{\mathbb{R}^{d}}\frac{1}%
{2}\langle Ax,x\rangle dP
\end{align*}
for each fixed $\alpha \in I$. Thus we get%
\[
\bar{G}(A)=\lim_{n\rightarrow \infty}n\mathbb{\bar{E}}\left[  \frac{1}%
{n}\left(  \frac{1}{2}\langle AX_{1},X_{1}\rangle \right)  \right]
=\inf_{\alpha \in I}\sup_{P\in \mathcal{P}_{\alpha}^{\ast}}\int_{\mathbb{R}^{d}%
}\frac{1}{2}\langle Ax,x\rangle dP
\]
for $A\in \mathbb{S}_{d}$. Thus, by Theorem \ref{th-12}, we obtain
\[
\lim_{n\rightarrow \infty}\mathbb{\bar{E}}\left[  \phi \left(  \frac{1}{\sqrt
{n}}\sum_{i=1}^{n}X_{i}\right)  \right]  =\bar{u}^{\phi}(1,0)\text{ for }%
\phi \in C_{b.Lip}(\mathbb{R}^{d}),
\]
where $\bar{u}^{\phi}$ is the unique viscosity solution of the following PDE:%
\[
\partial_{t}u-\bar{G}(D^{2}u)=0\text{, }u(0,x)=\phi(x).
\]

\end{example}

\begin{example}
Let $(B_{t})_{t\geq0}$ be a standard $d$-dimensional Brownian motion on a
complete probability space $(\Omega,\mathcal{F},P)$. The augmentation of
$(\mathcal{F}_{t}^{B})_{t\geq0}$ is denoted by $(\mathcal{F}_{t})_{t\geq0}$.
Set $\mathcal{H}=\cup_{n=1}^{\infty}L^{2}(\mathcal{F}_{n})$, where%
\[
L^{2}(\mathcal{F}_{n}):=\{ \xi \in \mathcal{F}_{n}:E[|\xi|^{2}]<\infty \}.
\]
Let $g:\mathbb{R}^{d}\rightarrow \mathbb{R}$ satisfy $g(0)=0$ and
$|g(z)-g(z^{\prime})|\leq \mu|z-z^{\prime}|$ for each $z$, $z^{\prime}%
\in \mathbb{R}^{d}$, where $\mu>0$ is a constant. For each $\xi \in \mathcal{H}$,
there exists an $n\geq1$ such that $\xi \in L^{2}(\mathcal{F}_{n})$. Consider
the following two backward stochastic differential equations (BSDEs):%
\[
\tilde{Y}_{t}=\xi+\int_{t}^{n}g(\tilde{Z}_{s})ds-\int_{t}^{n}\tilde{Z}%
_{s}dB_{s},
\]%
\[
Y_{t}=\xi+\int_{t}^{T}\mu|Z_{s}|ds-\int_{t}^{T}Z_{s}dB_{s}.
\]
Define%
\[
\mathbb{\tilde{E}}[\xi]=\tilde{Y}_{0}\text{ and }\mathbb{\hat{E}}[\xi]=Y_{0}.
\]
It is easy to check that the nonlinear expectation $\mathbb{\tilde{E}}[\cdot]$
is dominated by the sublinear expectation $\mathbb{\hat{E}}[\cdot]$ on
$(\Omega,\mathcal{H})$. Set $Y_{n}=B_{n}-B_{n-1}$ for $n\in \mathbb{N}$. It is
easy to verify that $\mathbb{\hat{E}}\left[  (|Y_{1}|-N)^{+}\right]
\rightarrow0$ as $N\rightarrow \infty$, $Y_{n+1}\overset{d}{=}Y_{1}$ and
$Y_{n+1}$ is independent of $(Y_{1}$,$\ldots$,$Y_{n})$ for $n\geq1$ under
$\mathbb{\tilde{E}}[\cdot]$ and $\mathbb{\hat{E}}[\cdot]$ respectively. By the
definition of $\mathbb{\tilde{E}}[\cdot]$, we have for each fixed
$p\in \mathbb{R}^{d}$,%
\[
n\mathbb{\tilde{E}}\left[  \frac{1}{n}\langle p,Y_{1}\rangle \right]
=n\tilde{Y}_{0}^{n},
\]
where%
\[
\tilde{Y}_{t}^{n}=\frac{1}{n}\langle p,B_{1}\rangle+\int_{t}^{1}g(\tilde
{Z}_{s}^{n})ds-\int_{t}^{1}\tilde{Z}_{s}^{n}dB_{s}.
\]
Set $g_{n}(z)=ng(z/n)$ for $z\in \mathbb{R}^{d}$. Then $g_{n}(0)=0$ and
$|g_{n}(z)-g_{n}(z^{\prime})|\leq \mu|z-z^{\prime}|$ for each $z$, $z^{\prime
}\in \mathbb{R}^{d}$. It is easy to verify that $(Y^{n},Z^{n})=(n\tilde{Y}%
^{n},n\tilde{Z}^{n})$ satisfies the following BSDE:%
\begin{equation}
Y_{t}^{n}=\langle p,B_{1}\rangle+\int_{t}^{1}g_{n}(Z_{s}^{n})ds-\int_{t}%
^{1}Z_{s}^{n}dB_{s}. \label{e3-16}%
\end{equation}
Suppose that
\begin{equation}
g_{n}(z)\rightarrow \bar{g}(z)\text{ as }n\rightarrow \infty \text{ for each
}z\in \mathbb{R}^{d}. \label{e3-15}%
\end{equation}
It is easy to check that $\bar{g}(0)=0$ and $|\bar{g}(z)-\bar{g}(z^{\prime
})|\leq \mu|z-z^{\prime}|$ for each $z$, $z^{\prime}\in \mathbb{R}^{d}$.
Consider the following BSDE:%
\begin{equation}
\bar{Y}_{t}=\langle p,B_{1}\rangle+\int_{t}^{1}\bar{g}(\bar{Z}_{s})ds-\int
_{t}^{1}\bar{Z}_{s}dB_{s}. \label{e3-17}%
\end{equation}
By the the standard estimate for BSDEs (\ref{e3-16}) and (\ref{e3-17}), we
obtain%
\[
|Y_{0}^{n}-\bar{Y}_{0}|^{2}\leq CE\left[  \int_{0}^{1}|g_{n}(\bar{Z}_{s}%
)-\bar{g}(\bar{Z}_{s})|^{2}ds\right]  ,
\]
where the constant $C$ only depends on $\mu$. Then we get $Y_{0}%
^{n}\rightarrow \bar{Y}_{0}$ as $n\rightarrow \infty$ by the dominated
convergence theorem. It is easy to verify that
\[
(\bar{Y}_{t},\bar{Z}_{t})=(\langle p,B_{t}\rangle+\bar{g}(p)(1-t),p)\text{ for
}t\in \lbrack0,1]
\]
satisfies (\ref{e3-17}), which implies%
\[
\tilde{G}(p)=\lim_{n\rightarrow \infty}Y_{0}^{n}=\bar{Y}_{0}=\bar{g}(p)\text{
for }p\in \mathbb{R}^{d}.
\]
Thus, under the assumption (\ref{e3-15}), we obtain
\[
\lim_{n\rightarrow \infty}\mathbb{\tilde{E}}\left[  \phi \left(  \frac{B_{n}}%
{n}\right)  \right]  =u^{\phi}(1,0)\text{ for }\phi \in C_{b.Lip}%
(\mathbb{R}^{d})
\]
by Theorem \ref{th-12}, where $u^{\phi}$ is the unique viscosity solution of
the following PDE:%
\[
\partial_{t}u-\bar{g}(Du)=0\text{, }u(0,x)=\phi(x).
\]

\end{example}

\  \  \

\end{document}